\documentclass[11pt,reqno]{amsart}
\usepackage{amsmath, amsthm, amssymb}
\usepackage{enumitem}
\usepackage[foot]{amsaddr}

\usepackage{fullpage}
\usepackage[many]{tcolorbox}
\usepackage{xcolor}
\usepackage{wasysym}
\usepackage{mathtools}
\usepackage{scalerel}
\usepackage{bbm}
\usepackage[all]{xy}
\usepackage{mathrsfs}
\usepackage{adjustbox}
\usepackage{tensor}

\usepackage{tikz}
\usepackage{adjustbox}
\usetikzlibrary{arrows,backgrounds,patterns.meta}
\usetikzlibrary{positioning,shadings,cd}
\usetikzlibrary{shapes}
\usetikzlibrary{backgrounds}
\usetikzlibrary{decorations,decorations.pathreplacing,decorations.markings,decorations.pathmorphing}
\tikzstyle{snake}=[decorate, decoration={snake, segment length=1mm, amplitude=.5mm}]
\usetikzlibrary{fit,calc,through}
\usetikzlibrary{external}
\newcommand{\tikzmath}[2][]
{\vcenter{\hbox{\begin{tikzpicture}[#1]#2\end{tikzpicture}}}
}

\tikzset{super thick/.style={line width=3pt}}
\tikzstyle{far>}=[decoration={markings, mark=at position 0.75 with {\arrow{>}}}, postaction={decorate}]
\tikzstyle{mid>}=[decoration={markings, mark=at position 0.55 with {\arrow{>}}}, postaction={decorate}]
\tikzstyle{mid<}=[decoration={markings, mark=at position 0.55 with {\arrow{<}}}, postaction={decorate}]
\tikzset{super thick/.style={line width=3pt}}
\tikzstyle{far>}=[decoration={markings, mark=at position 0.75 with {\arrow{>}}}, postaction={decorate}]
\tikzstyle{mid>}=[decoration={markings, mark=at position 0.55 with {\arrow{>}}}, postaction={decorate}]
\tikzstyle{mid<}=[decoration={markings, mark=at position 0.55 with {\arrow{<}}}, postaction={decorate}]
\tikzstyle{knot}=[preaction={super thick, white, draw}]
\tikzstyle{coupon}=[draw, very thick, rectangle, rounded corners=5pt]
\tikzset{Rightarrow/.style={double equal sign distance,>={Implies},->},
triplecd/.style={-,preaction={draw,Rightarrow}},
quadruplecd/.style={preaction={draw,Rightarrow,
shorten >=0pt
},
shorten >=1pt,
-,double,double
distance=0.2pt}}
\tikzset{
    tripleline/.style args={[#1] in [#2] in [#3]}{
        #1,preaction={preaction={draw,#3},draw,#2}
    }
}
\tikzstyle{triple}=[tripleline={[line width=.15mm,black] in
      [line width=.7mm,white] in
      [line width=1mm,black]}] 
\tikzset{
    quadrupleline/.style args={[#1] in [#2] in [#3] in [#4]}{
        #1,preaction={preaction={preaction={draw,#4},draw,#3}, draw,#2}
    }
}
\tikzstyle{quadruple}=[quadrupleline={[line width=.3mm,white] in
      [line width=.6mm,black] in
      [line width=1.2mm,white] in
      [line width=1.5mm,black]}]

\definecolor{violet}{RGB}{148,0,211}
\definecolor{DarkGreen}{RGB}{0,150,0}
\definecolor{rufous}{HTML}{A81C07}

\usepackage[plainpages=false,hypertexnames=false,pdfpagelabels]{hyperref}
\definecolor{medium-blue}{rgb}{0,0,.8}
\hypersetup{colorlinks, linkcolor={purple}, citecolor={medium-blue}, urlcolor={medium-blue}}
\newcommand{\arxiv}[1]{\href{http://arxiv.org/abs/#1}{\tt arXiv:\nolinkurl{#1}}}
\newcommand{\arXiv}[1]{\href{http://arxiv.org/abs/#1}{\tt arXiv:\nolinkurl{#1}}}
\newcommand{\mathscinet}[1]{\href{http://www.ams.org/mathscinet-getitem?mr=#1}{\tt #1}}
\newcommand{\doi}[1]{\href{http://dx.doi.org/#1}{{\tt DOI:#1}}}

\newcommand{\googlebooks}[1]{(preview at \href{https://books.google.com/books?id=#1}{google books})}

\DeclareMathOperator{\Ad}{Ad}

\DeclareMathOperator{\End}{End}

\DeclareMathOperator{\Hom}{Hom}
\DeclareMathOperator{\id}{id}
\DeclareMathOperator{\Id}{Id}

\DeclareMathOperator{\Irr}{Irr}

\DeclareMathOperator{\tr}{tr}

\renewcommand{\Im}{\operatorname{Im}}

\newcommand{\set}[2]{\left\{#1 \middle| #2\right\}}

\newcommand{\Hilb}{\mathsf{Hilb}}

\newcommand{\SSS}{\mathsf{SSS}}

\def\semicolon{;}
\def\applytolist#1{
    \expandafter\def\csname multi#1\endcsname##1{
        \def\multiack{##1}\ifx\multiack\semicolon
            \def\next{\relax}
        \else
            \csname #1\endcsname{##1}
            \def\next{\csname multi#1\endcsname}
        \fi
        \next}
    \csname multi#1\endcsname}

\def\calc#1{\expandafter\def\csname c#1\endcsname{{\mathcal #1}}}
\applytolist{calc}QWERTYUIOPLKJHGFDSAZXCVBNM;
\def\bbc#1{\expandafter\def\csname bb#1\endcsname{{\mathbb #1}}}
\applytolist{bbc}QWERTYUIOPLKJHGFDSAZXCVBNM;
\def\bfc#1{\expandafter\def\csname bf#1\endcsname{{\mathbf #1}}}
\applytolist{bfc}QWERTYUIOPLKJHGFDSAZXCVBNM;
\def\sfc#1{\expandafter\def\csname s#1\endcsname{{\sf #1}}}
\applytolist{sfc}QWERTYUIOPLKJHGFDSAZXCVBNM;
\def\fc#1{\expandafter\def\csname f#1\endcsname{{\mathfrak #1}}}
\applytolist{fc}QWERTYUIOPLKJHGFDSAZXCVBNM;
\def\rmc#1{\expandafter\def\csname rm#1\endcsname{{\mathrm #1}}}
\applytolist{rmc}QWERTYUIOPLKJHGFDSAZXCVBNM;
\def\scrc#1{\expandafter\def\csname scr#1\endcsname{{\mathscr #1}}}
\applytolist{scrc}QWERTYUIOPLKJHGFDSAZXCVBNM;

\numberwithin{equation}{section}

\theoremstyle{plain}
\newtheorem{theorem}[equation]{Theorem}
\newtheorem*{thm*}{Theorem}
\newtheorem{corollary}[equation]{Corollary}
\newtheorem{lemma}[equation]{Lemma}
\newtheorem{proposition}[equation]{Proposition}

\newtheorem*{claim*}{Claim}

\newtheorem{thmalpha}{Theorem}

\newtheorem{definitionproposition}[equation]{Definition-Proposition}

\theoremstyle{definition}
\newtheorem{definition}[equation]{Definition}

\newtheorem*{trick*}{Trick}
\newtheorem{construction}[equation]{Construction}

\newtheorem{fact}[equation]{Fact}

\newtheorem{remark}[equation]{Remark}

\title{Superselection theory for 2D braided quantum spin systems via Connes fusion}

\author[G.~Faurot]{Gregory Faurot}
\author[C.~Li]{Charlton Li}
\author[D.~Penneys]{David Penneys}
\author[E.~Sherman]{Emeline Sherman}

\begin{document}

\begin{abstract}
The article [\arxiv{2410.21454}] constructed a braided $\mathrm{W^*}$-tensor category of superselection sectors associated to a net of von Neumann algebras on a suitably geometric poset using localized and transportable endomorphisms.
In this article, we construct an equivalent tensor product and braiding using the Connes fusion tensor product without localizing, using ideas from [\arxiv{1812.04470}] for conformal nets.
As an application, we prove that the category of superselection sectors associated to the 2D braided quantum spin system constructed from a unitary braided fusion category $\mathcal{B}$ in [\arxiv{2506.19969}] is equivalent to $\mathsf{Hilb}(\mathcal{B})^{\rm op}$, the opposite of the unitary ind-completion of $\mathcal{B}$, as a braided $\mathrm{W^*}$-tensor category.
\end{abstract}

\maketitle

\section{Introduction}
The Haag-Kastler approach to algebraic quantum field theory assigns local algebras $\fA_R$ of observables to regions $R$ of spacetime \cite{MR165864}.
The $\rmC^*$ inductive limit $\fA:=\varinjlim \fA_R$ over the net of local algebras is called the \emph{quasi-local algebra}, and usually comes equipped with a standard or vacuum representation.
The Doplicher-Haag-Roberts theory of \emph{superselection sectors} \cite{MR0297259} studies representations of $\mathfrak{A}$ which satisfy a certain \emph{superselection criterion}: they are equivalent to the vacuum representation outside of any distinguished double cone.
This criterion is used to show that such admissible representations are equivalent to \emph{localized} and \emph{transportable} endomorphisms of $\mathfrak{A}$. 
Using these endomorphisms allows for the construction of a braided $\rmC^*$-tensor category, under enough geometric structure of spacetime.

The DHR framework has found applications in low-dimensional conformal field theory \cite{MR660538,MR1231644} and topologically ordered quantum spin systems \cite{MR2804555,MR4362722} to construct braided $\rmC^*$-tensor categories of superselection sectors.
Both of these approaches were simultaneously generalized in \cite{MR4927814} to the setting of a net of von Neumann algebras associated to an involutive poset.\footnote{Another generalization based on $\rmC^*$-prefactorization algebras appeared recently in \cite{MR4998903}.}
That is, we start with a poset $\cP$ equipped with an order-reversing involution $p\mapsto p'$ (satisfying $p\leq q$ implies $q'\leq p'$ and $p''=p$ for all $p,q\in \cP$), together with a vacuum Hilbert space $H$ and an assignment of a von Neumann algebra $A_p\subset B(H)$ for every $p\in \cP$.
This net of von Neumann algebras is required to satisfy three axioms:
\begin{itemize}
\item 
(isotony)
$p\leq q$ implies $A_p\subseteq A_q$,
\item 
(Haag duality)
$A_p' = A_{p'}$ for all $p\in \cP$, and 
\item 
(absorbing) the vacuum $H$ is \emph{absorbing} for each $A_p$, i.e., for every other $A_p$-module $K$, $H\oplus K\cong H$ as $A_p$-modules.\footnote{All Hilbert spaces in this article are assumed to be separable, and all modules are assumed to be normal.}
\end{itemize}
The final absorbing property is equivalent in this context to all the $A_p$ being properly infinite.

A \emph{superselection sector} of $(A_p)_{p\in \cP}$ is then another Hilbert space $K$ equipped with a family of representations $\pi_p: A_p\to B(K)$ satisfying:
\begin{itemize}
\item 
(isotony)
$\pi_q|_p=\pi_q$ whenever $p\leq q$,
\item 
(locality)
$\pi(A_{p'})\subset \pi(A_p)'$ for all $p\in \cP$, and
\item 
(absorbing) $K$ is absorbing for each $A_p$.
\end{itemize}
The final absorbing property is equivalent to each $\pi_p$ being injective; its purpose is to ensure that a superselection sector can be \emph{localized} in any $p\in \cP$, satisfying $K=H$ and $\pi_{p'}=\id_{A_{p'}}$.
If moreover $\cP$ satisfies certain \emph{geometric axioms} (see \S\ref{sec:GeometricAxioms} below), by considering the category $\SSS_p$ of superselection sectors localized in $p$, one can construct a tensor product $\circ_p$ and a braiding, endowing $\SSS_p$ with the structure of a braided $\rmW^*$-tensor category.
One then shows that localizing in any $p\in\cP$ produces an equivalent braided $\rmW^*$-tensor category.

For concrete applications, it can be highly disadvantageous to perform localization in order to compute the braiding and tensor product.
For example, \cite{MR1645078} used Connes fusion (relative tensor product) \cite{MR1303779,MR703809} to compute the fusion for the positive energy representations of the loop group of $SU(N)$, a special case of the WZW model (see \cite{MR1152377}).
Connes fusion was further implemented to compute the fusion and braiding of superselection sectors of conformal nets in \cite{MR4239831}, and it was used to compute the canonical twists associated to a conformal net in \cite{2605.18446}.

In this article, we adapt the Connes fusion approach of \cite{MR4239831} in the setting of conformal nets to the setting of general nets of von Neumann algebras.
Just as the Connes fusion of bimodules ${}_AH_B, {}_BK_C$ of von Neumann algebras can be computed in three equivalent ways by completing one of the following spaces:
\begin{itemize}
\item 
$\Hom(L^2B_B, H_B) \otimes_B K$,
\item 
$\Hom(L^2B_B, H_B) \otimes_B L^2B \otimes_B \Hom({}_BL^2B, {}_BK)$, or
\item 
$H\otimes_B \Hom({}_BL^2B, {}_BK)$,
\end{itemize}
so can the fusion of superselection sectors.
Using the same setup as \cite{MR4927814}, we construct three different versions of fusion by completing the spaces:
\begin{itemize}
\item $\Hom_{A_{p'}}(H,K)\otimes_{A_p} L$
\item $\Hom_{A_{p'}}(H, K)\otimes_{A_p} H\otimes_{A_q^{\rm op}} \Hom_{A_{q'}}(H, L)$
\item $K \otimes_{A_q^{\rm op}} \Hom_{A_{q'}}(H, L)$
\end{itemize}
where $K$ and $L$ are superselection sectors.\footnote{When the Hilbert space appears to the left of the hom space, we must take a relative product over the opposite von Neumann algebra as the Hilbert space is a left module.}
These spaces are all unitarily isomorphic, and we use all three of them, together with various geometric properties of our poset $\cP$, to construct an action of the whole net $(A_p)_{p\in \cP}$ on them.
(These unitary isomorphisms are compatible with the actions of the net of algebras.)
Focusing on the first monoidal structure, which we denote by $\boxtimes_p$, we use similar techniques to equip $\SSS$ with a braiding, endowing $(\SSS,\boxtimes_p)$ with the structure of a braided $\rmW^*$-tensor category.

The essential thing to check is that our new definition agrees is equivalent with the old definition.

\begin{thmalpha} \label{thm:A}
    $(\SSS, \boxtimes_p)$ and $(\SSS_p, \circ_p)$ are equivalent braided $\rmW^*$-tensor categories.
\end{thmalpha}

Our main reason for performing the above construction is to compute the fusion and braiding for the \emph{2D braided quantum spin systems} constructed in \cite{2506.19969} from a unitary braided fusion category (UBFC) equipped with a strong tensor generator $X\in\cB$.\footnote{We call $X\in \cB$ a \emph{strong tensor generator} if there is an $n>0$ such that every simple object of $\cB$ is isomorphic to a summand of $X^{\otimes n}$.}
In more detail, the 2D braided quantum spin system assigns 
a canonical object $X^{\otimes R}\in \cB$ to every \emph{disk-like region} $R$ in a $\bbZ^2$ lattice,\footnote{For a region $R\subset \bbZ^2$, we write $[R]=\bigcup_{v\in R} [v]$ where $[v]$ is the closed $\ell^\infty$-ball at $v$ of radius $1/2$.
We say $R$ is \emph{disk-like} if $\partial[R]\subset \bbR^2$ is homeomorphic to $S^1$.}
from which we define the local operator algebra
$\fA_R:=\End_\cB(X^{\otimes R})$.
Inclusions of disk-like regions give inclusions of the operator algebras $\fA_R$ by tensoring with identity strands, where the braiding of $\cB$ is used in a crucial way to make the inclusion well-defined.
\begin{equation}
\label{eq:InclusionOfDiskLikeRegions}
\tikzmath{
\foreach \y in {0,.5,...,2.5}{
\foreach \x in {0,.5,...,2.5}{
\filldraw (\x,\y) circle (.05cm);
}}
\draw[thick, blue] (.25,.25) -- (.75,.25) -- (.75,.75) -- (1.75,.75) -- (1.75,.25) -- (2.25,.25) -- (2.25,2.25) -- (1.75,2.25) -- (1.75,1.75) -- (.75,1.75) -- (.75,2.25) -- (.25,2.25) -- (.25,.25);
\node[blue] at (1.25,1.25) {$R$};
\draw[thick, red] (-.25,-.25) rectangle (2.75,2.75);
\node[red] at (1.25,.25) {$S$};
}
\qquad\qquad\rightsquigarrow\qquad\qquad
\tikzmath{
\foreach \x in {.5,1,1.5,2}{
\draw[thick,red] ($ (\x,2.5) + (-.15,.35)$) -- ($ (\x,2.5) + (.15,-.35)$);
}
\foreach \x in {1,1.5}{
\draw[thick,red] ($ (\x,2) + (-.15,.35)$) -- ($ (\x,2) + (.15,-.35)$);
}
\foreach \x in {0,2.5}{
\foreach \y in {0,.5,...,2.5}{
\draw[thick,red] ($ (\x,\y) + (-.15,.35)$) -- ($ (\x,\y) + (.15,-.35)$);
}}
\foreach \y in {.5,1,1.5,2}{
\foreach \x in {.5,2}{
\draw[thick,blue] (\x,\y) -- ($ (\x,\y) + (.15,-.35)$);
}}
\foreach \y in {1,1.5}{
\foreach \x in {1,1.5}{
\draw[thick,blue] (\x,\y) -- ($ (\x,\y) + (.15,-.35)$);
}}
\fill[fill=gray!60, opacity=.8] (.25,.25) -- (.75,.25) -- (.75,.75) -- (1.75,.75) -- (1.75,.25) -- (2.25,.25) -- (2.25,2.25) -- (1.75,2.25) -- (1.75,1.75) -- (.75,1.75) -- (.75,2.25) -- (.25,2.25) -- (.25,.25);
\draw[thick, blue] (.25,.25) -- (.75,.25) -- (.75,.75) -- (1.75,.75) -- (1.75,.25) -- (2.25,.25) -- (2.25,2.25) -- (1.75,2.25) -- (1.75,1.75) -- (.75,1.75) -- (.75,2.25) -- (.25,2.25) -- (.25,.25);
\foreach \y in {.5,1,1.5,2}{
\foreach \x in {.5,2}{
\draw[thick,blue] (\x,\y) -- ($ (\x,\y) + (-.15,.35)$);
}}
\foreach \y in {1,1.5}{
\foreach \x in {1,1.5}{
\draw[thick,blue] (\x,\y) -- ($ (\x,\y) + (-.15,.35)$);
}}
\node[blue] at (1.25,1.25) {$f$};
\foreach \x in {.5,1,1.5,2}{
\draw[thick,red,knot] ($ (\x,0) + (-.15,.35)$) -- ($ (\x,0) + (.15,-.35)$);
}
\foreach \x in {1,1.5}{
\draw[thick,red,knot] ($ (\x,.5) + (-.15,.35)$) -- ($ (\x,.5) + (.15,-.35)$);
}
\node at (1.25,-.65) {$f\in \fA_R\subset \fA_S$};
}
\end{equation}
Taking $X=\bigoplus_{b\in \Irr(\cB)}b$, the sum over all the simples of $\cB$, we get a canonical state on each $\fA_R$ given by
$\psi(f):= (i^\dag)^{\otimes R}\circ f \circ i^{\otimes R}$, where $i:1_\cB\hookrightarrow X$ is an isometry.
One then defines a net of von Neumann algebras on the poset of cones in $\bbR^2$ 
by setting $\fA_\Lambda:=\varinjlim_{R\subset \Lambda} \fA_R$
and $A_\Lambda :=\fA_\Lambda''$ acting on $L^2(\fA,\psi)$, called the \emph{2D braided categorical net}.

In \cite[Thm.~4.8]{2506.19969}, it was proven that the 
underlying $\rmW^*$-category (without the fusion or braiding) of the superselection sectors of the net $(A_\Lambda)$ is equivalent to $\Hilb(\cB)^{\rm op}$, the opposite category of the unitary ind-completion of $\cB$ \cite{MR3509018,MR3687214}.
The (abstract of v1 of the) article \cite{2506.19969} conjectured that the fusion and braiding should also capture the fusion and braiding of $\cB$.
We use our new Connes fusion description of the tensor product and braiding of superselection sectors to verify this conjecture. 

\begin{thmalpha} \label{thm:B}
    Let $\cB$ be a unitary braided fusion category with strong tensor generator $X \in \cB$.
    There is a unitary tensorator which extends the
    $\rmW^*$-equivalence $\Hilb(\cB)^{\rm op}\to \mathsf{SSS}(A)$ from \cite{2506.19969} to a braided monoidal equivalence, where $A=(A_\Lambda)$ is the 2D braided categorical net of von Neumann algebras associated to $(\cB,X)$.
\end{thmalpha}

\subsection*{Acknowledgments}

This article is the undergraduate research project of Charlton Li from OSUIM during Summer 2026, supervised by Gregory Faurot, Emeline Sherman, and David Penneys.
The program was supported by the Mathematics Department of The Ohio State University along with Penneys' NSF DMS grants 2154389 and 2554723.
The authors would like to thank 
Andr\'e Henriques,
Corey Jones,
Boris Kj{\ae}r,
Adri\`a Mar\'in-Salvador,
Pieter Naaijkens,
and
Daniel Wallick
for helpful conversations. 

\subsection*{AI use statement}
No LLMs were used for this manuscript nor any of the results it contains.

\section{Background}\label{sec:background}

In this section, we review the necessary background material from \cite{MR4927814} on geometric axioms for involutive posets, nets of von Neumann algebras, and superselection sectors.
Recall that an \emph{involutive poset} is a poset $\cP$ with an  order-reversing involution $p\mapsto p'$, i.e., $p=p''$ and $p\leq q$ implies $q'\leq p'$ for all $p,q\in \cP$.
We write
\[
p\Cap q:= \set{r\in \cP}{r\leq p\text{ and }r\leq q}.
\]
We say $p,q$ are \emph{disjoint} if $p\Cap q = \emptyset$.

\subsection{Geometric axioms for posets}
\label{sec:GeometricAxioms}
In order to later construct a monoidal product and braiding on the category of superselection sectors, we work with an involutive poset which satisfies some additional geometric axioms.
\begin{enumerate}[label=\textup{(GA\arabic*)}, series=geom]
\setcounter{enumi}{-1}
\item 
\label{geom:SelfDisjoint}
\underline{Disjoint complements:}
For every $p\in\cP$, $p\Cap p'=\emptyset$.

\item 
\label{geom:saturate}
\underline{Saturation:}
Every $p\in\cP$ is \emph{saturated}: for any $q\in\cP$, $p\Cap q\neq \emptyset$ or $p\Cap q'\neq \emptyset$ (possibly both).

\item 
\label{geom:splitting}
\underline{Splitting:}
Every $p \in \cP$ \emph{splits}: there are disjoint $r,s \leq p$.
\end{enumerate}
An important tool for the study of the superselection category in \cite{MR4927814} is the notion of a zig-zag. 
\begin{definition}
    A \emph{zig-zag} from $p$ to $q$ is a sequence of elements $(z_1,y_1,\dots , z_n, y_n, z_{n+1}) \subset \cP$ with $z_1=p$, $z_{n+1}=q$, and $z_j, z_{j+1}\leq y_j$ for all $j$. We will often write such a zig-zag as
    $$z_1\leq y_1 \geq z_2 \leq \cdots \geq z_n \leq y_n \geq z_{n+1}.$$
\end{definition}
\noindent As a consequence of \ref{geom:saturate} and \ref{geom:splitting}, $\cP$ is connected in the sense that there is a zig-zag between any two elements in $\cP$ \cite[Lemma~1.1]{MR4927814}.

\begin{definition}
    For $p, q \in \cP$, we say $p$ is $q$-\textit{small} if there exist $r, s \in \cP$ such that $p, q \leq r$ and $p, q' \leq s$. Call $p$ a $q$-\textit{indicator} if $p\leq q$ or $p\leq q'$.
\end{definition}

\noindent Another consequence of \ref{geom:saturate} and \ref{geom:splitting} is the following.
\begin{lemma} [\cite{MR4927814}, Lemma 1.2]
    The axioms \ref{geom:saturate} and \ref{geom:splitting} are equivalent to the axiom:
\begin{enumerate}[label=\textup{(GA1.5)}, series=geomconsequence]
\item
\label{geom:qSmallqIndicator}
\textup{\underline{Small indicators:}
For any $p,q\in\cP$, there is a $q$-small $q$-indicator $\widetilde{p}\leq p$.
}
\end{enumerate}
\end{lemma}

\noindent In addition to the previous geometric axioms, we ask for the existence of zig-zags between small indicators which alternate between small elements and indicators.

\begin{enumerate}[label=\textup{(GA\arabic*)}, series=geom]
\setcounter{enumi}{2}
\item
\label{geom:ZigZag}
\underline{Zig-zag:}
For every $p,q\in\cP$ and any two choices of $q$-small $q$-indicator $\widetilde{p},\widehat{p}\leq p$ as in \ref{geom:qSmallqIndicator},
there is a zig-zag between $\widetilde{p},\widehat{p}$ contained in $p$ such that each $z_j$ is a $q$-indicator and each $y_j$ is $q$-small (and hence each $z_j$ as well).
\end{enumerate}

\noindent As discussed in \cite{MR4927814}, these geometric axioms hold in many cases of interest, including intervals in $S^1$, cones in $\mathbb R^2$ and in $\mathbb R^3$, and disks in $S^2$. 
The authors of \cite{MR4927814} also prove that \ref{geom:SelfDisjoint}-\ref{geom:ZigZag} imply the existence of mutually disjoint zig-zags and reflections.

\begin{itemize}
\item 
    Given $p_1,p_2,q_1,q_2\in \cP$ with $p_1\leq q_1'$ and $p_2\leq q_2'$, a \emph{mutually disjoint zig-zag} $(p_1,q_1) \leftrightsquigarrow (p_2,q_2)$ is a pair of zig-zags
    \begin{align*}
        &(p_1=x_1,\textcolor{blue}{w_1}, \textcolor{red}{x_2}, \dots, x_n, \textcolor{blue}{w_n}, \textcolor{red}{x_{n+1}}=p_2)\\
        &(q_1=\textcolor{blue}{z_1}, \textcolor{red}{y_1}, z_2, \dots, \textcolor{blue}{z_n}, \textcolor{red}{y_n}, z_{n+1}=s_2)
    \end{align*}
    such that $\textcolor{blue}{w_i}$ is disjoint from $\textcolor{blue}{z_i}$ and $\textcolor{red}{y_i}$ is disjoint from $\textcolor{red}{x_{i+1}}$ for $i=1,\dots, n$.
\item 
A \emph{reflection} of a splitting $(r,s)$ of $p$ is a splitting $(a,b)$ of $p'$ with some $c \in \cP$ such that $a,r \leq c$ and $b,s \leq c'$.
\end{itemize}

\begin{enumerate}[label=\textup{(GA\arabic*)}, series=geom]
\setcounter{enumi}{3}
\item
\label{geom:disjointZigZag}
\underline{Mutually disjoint zig-zag:}
    For any two splittings $(r_1,s_1)$ and $(r_2,s_2)$ of $p$, there is a mutually disjoint zig-zag $(r_1,s_1) \leftrightsquigarrow(r_2,s_2)$ or $(r_1,s_1) \leftrightsquigarrow(s_2,r_2)$ contained in $p$ \cite[Theorem~4.19]{MR4927814}.
    \end{enumerate}

\begin{enumerate}[label=\textup{(GA\arabic*)}, series=geom]
\setcounter{enumi}{4}
\item 
\label{geom:reflection}
\underline{Reflection:} Every $p\in \cP$ admits a splitting which has a reflection \cite[Proposition~4.11]{MR4927814}.
\end{enumerate}

\subsection{Nets of von Neumann algebras and superselection sectors}

Having described the involutive posets we shall be working with, we now give the requisite background on $\cP$-nets of algebras.
In what follows, we assume all von Neumann algebras have separable preduals.

\begin{definition}
    Given an involutive poset $\cP$, a $\cP$-\emph{net of algebras} is a family of von Neumann algebras $\set{A_p\subseteq B(\cH)}{p \in \cP}$ satisfying:
    \begin{itemize}
        \item (isotone) $p\leq q$ implies $A_p\subseteq A_q$,
        \item (Haag duality) $A_{p'}=A_p'$ for all $p \in \cP$, and
        \item (absorbing) for each $p \in \cP$, $\cH$ is an absorbing representation of $A_p$.
    \end{itemize}
\end{definition}
Here, a representation $\cH$ of $A$ is \emph{absorbing} if $\cH \oplus \cK$ and $\cH$ are isomorphic representations of $A$ for any normal representation $\cK$ of $A$. Absorbing representations are unique up to unitary equivalence. 
As in \cite[Cor.~2.12]{MR4927814}, the absorbing property may be equivalently replaced with the axiom that each $A_p$ is properly infinite.

\begin{definition}
    A \emph{superselection sector} of a $\cP$-net of algebras is a family of normal representations $\pi_p\colon A_p \to B(\cK)$ such that
    \begin{itemize}
        \item (isotone) if $p\leq q$, then ${\left.{\pi_q}\right|_{A_p}}=\pi_p$,
        \item (local) if $p$ and $q$ are disjoint, then $[\pi_p(A_p),\pi_q(A_q)]=0$, and
        \item (absorbing) each representation $\pi_p$ is absorbing.
    \end{itemize}
\end{definition}
\noindent The absorbing property may be equivalently replaced by the assumption that each $\pi_p$ is injective. 
This follows from the fact that each $A_p$ is properly infinite \cite[Lemmas~2.11 and 2.18]{MR4927814}.

\begin{definition}
    For a $\cP$-net of algebras, the category $\SSS$ of superselection sectors is the category whose objects are superselection sectors and whose morphisms are bounded intertwiners, i.e., bounded linear operators $T \colon (\cK,\pi) \to (\cL, \sigma)$ such that $T\pi_p(x)=\sigma_p(x)T$ for all $x \in A_p$ and all $p \in \cP$.
\end{definition}
\noindent We will give an overview of the fusion and braiding constructed in \cite{MR4927814} in Section~\ref{sec:equivalencewithendomorphisms} below.

\subsection{Example: 2D braided quantum spin systems}
As an explicit example of a net of algebras, we sketch the construction from \cite[\S4]{2506.19969} of a \emph{2D braided quantum spin system}
from a unitary braided fusion category $\cB$.
The involutive poset will be the poset of cones in $\bbR^2$.

\subsubsection*{The net of algebras}
Fix an object $X \in \cB$ that contains $1_{\cB}$ as a subobject, together with an isometry $i_X\colon 1_{\cB} \to X$.
We assign a copy of $X$ to each vertex of the $\mathbb Z^2$ lattice in $\mathbb R^2$. 
To each bounded `disk-like region' \cite[Def.~2.4]{2605.10693} $R\subset \mathbb R^2$, there is an associated object $X^{\otimes R}\in \cB$ given by taking the monoidal product of the copies of $X$ associated to the vertices inside $R$.\footnote{The braiding is used in an essential way to prove that $X^{\otimes R}$ is well-defined.
For an approach using operads, we refer the reader to \cite{MR1989873}.
}
To each disk-like region $R$, we assign the finite-dimensional C$^*$-algebra 
$$\mathfrak A(R)\coloneq \End_{\cB}(X^{\otimes R}).$$
If $R \subset S$ is an inclusion of bounded disk-like regions, we obtain an inclusion of finite-dimensional C$^*$-algebras $\mathfrak A(R)\hookrightarrow \mathfrak A(S)$ by tensoring each endomorphism with $1_X^{\otimes S\setminus R}$ as in \eqref{eq:InclusionOfDiskLikeRegions} in the introduction. 

We thus get a net of finite dimensional $\rmC^*$-algebras associated to the disk-like regions in $\bbR^2$.
For an infinite cone $\Lambda$, we obtain an AF $\rmC^*$-algebra $\fA(\Lambda)$ by taking the inductive limit over disk-like regions inside $\Lambda$.\footnote{Here, there is a subtle technicality about the poset of cones in $\bbR^2$ in that we want that the boundary $\partial[\Lambda]$ to be homeomorphic to $\bbR$.
We thus work with a modified poset of truncated cones as in \cite[\S2.1]{2605.10693}.
}

\subsubsection*{Superselection sectors}
To each disk-like region $R$, we define a (finite-dimensional) Hilbert space for each simple object $x \in \cB$
$$
\cK_x(R)\coloneq \cB(x, X^{\otimes R})
$$
with inner product $\langle f | g\rangle_{\cK_x}\coloneq \tr_{\cB}(f^\dagger \circ g).$ 
If $R \subset S$ is an inclusion of bounded disk-like regions, we obtain an isometric inclusion $\cK_x(R) \hookrightarrow \cK_x(S)$ by tensoring each morphism with $i_X^{\otimes S\setminus R}$. 
\begin{equation}
\label{eq:InclusionOfDiskLikeRegionsHilbertSpace}
\tikzmath{
\foreach \y in {0,.5,...,2.5}{
\foreach \x in {0,.5,...,2.5}{
\filldraw (\x,\y) circle (.05cm);
}}
\draw[thick, blue] (.25,.25) -- (.75,.25) -- (.75,.75) -- (1.75,.75) -- (1.75,.25) -- (2.25,.25) -- (2.25,2.25) -- (1.75,2.25) -- (1.75,1.75) -- (.75,1.75) -- (.75,2.25) -- (.25,2.25) -- (.25,.25);
\node[blue] at (1.25,1.25) {$R$};
\draw[thick, red] (-.25,-.25) rectangle (2.75,2.75);
\node[red] at (1.25,.25) {$S$};
}
\qquad\qquad\rightsquigarrow\qquad\qquad
\tikzmath{
\draw[snake, cyan, thick] (1.25,1.25) -- (3,-.25) node[right]{$\scriptstyle x$};
\foreach \x in {.5,1,1.5,2}{
\draw[thick,red] ($ (\x,2.5) + (-.15,.35)$) -- (\x,2.5);
\filldraw[red] (\x,2.5) circle (.05cm);
}
\foreach \x in {1,1.5}{
\draw[thick,red] ($ (\x,2) + (-.15,.35)$) -- (\x,2);
\filldraw[red] (\x,2) circle (.05cm);
}
\foreach \x in {0,2.5}{
\foreach \y in {0,.5,...,2.5}{
\draw[thick,red,knot] ($ (\x,\y) + (-.15,.35)$) -- (\x,\y);
\filldraw[red] (\x,\y) circle (.05cm);
}}
\fill[fill=gray!60, opacity=.8] (.25,.25) -- (.75,.25) -- (.75,.75) -- (1.75,.75) -- (1.75,.25) -- (2.25,.25) -- (2.25,2.25) -- (1.75,2.25) -- (1.75,1.75) -- (.75,1.75) -- (.75,2.25) -- (.25,2.25) -- (.25,.25);
\draw[thick, blue] (.25,.25) -- (.75,.25) -- (.75,.75) -- (1.75,.75) -- (1.75,.25) -- (2.25,.25) -- (2.25,2.25) -- (1.75,2.25) -- (1.75,1.75) -- (.75,1.75) -- (.75,2.25) -- (.25,2.25) -- (.25,.25);
\foreach \y in {.5,1,1.5,2}{
\foreach \x in {.5,2}{
\draw[thick,blue] (\x,\y) -- ($ (\x,\y) + (-.15,.35)$);
}}
\foreach \y in {1,1.5}{
\foreach \x in {1,1.5}{
\draw[thick,blue] (\x,\y) -- ($ (\x,\y) + (-.15,.35)$);
}}
\node[blue] at (1.25,1.25) {$g$};
\foreach \x in {.5,1,1.5,2}{
\draw[thick,red,knot] ($ (\x,0) + (-.15,.35)$) -- (\x,0);
\filldraw[red] (\x,0) circle (.05cm);
}
\foreach \x in {1,1.5}{
\draw[thick,red,knot] ($ (\x,.5) + (-.15,.35)$) -- (\x,.5);
\filldraw[red] (\x,.5) circle (.05cm);
}
\node at (1.25,-.65) {$g\in \cK_x(R)\subset \cK_x(S)$};
}
\end{equation}
By taking the inductive limit over an increasing sequence of disk-like regions that fill the plane, we obtain the total Hilbert space $\cK_x\coloneq \varinjlim \cK_x(R)$. 
There is also an obvious action of $
\fA(\Lambda)$ on each $\cK_x$ induced by post-composition on finite regions.

Next, for each bounded disk-like region $R$ there is a state $\psi_R$ on $\mathfrak A(R)$ given by $\psi_R(f)\coloneq (i_X^{\otimes R})^\dagger \circ f \circ (i_X^{\otimes R})$. 
This is compatible with inclusions of regions, and so defines a state on the entire quasi-local algebra $\fA$. 
\begin{equation}
\label{eq:InclusionOfDiskLikeRegionsForState}
\tikzmath{
\foreach \y in {0,.5,...,2.5}{
\foreach \x in {0,.5,...,2.5}{
\filldraw (\x,\y) circle (.05cm);
}}
\draw[thick, blue] (.25,.25) -- (.75,.25) -- (.75,.75) -- (1.75,.75) -- (1.75,.25) -- (2.25,.25) -- (2.25,2.25) -- (1.75,2.25) -- (1.75,1.75) -- (.75,1.75) -- (.75,2.25) -- (.25,2.25) -- (.25,.25);
\node[blue] at (1.25,1.25) {$R$};
\draw[thick, red] (-.25,-.25) rectangle (2.75,2.75);
\node[red] at (1.25,.25) {$S$};
}
\qquad\qquad\rightsquigarrow\qquad\qquad
\tikzmath{
\foreach \x in {.5,1,1.5,2}{
\draw[thick,red,{Circle[scale=0.7]}-{Circle[scale=0.7]}] ($ (\x,2.5) + (-.15,.35)$) -- ($ (\x,2.5) + (.15,-.35)$);
}
\foreach \x in {1,1.5}{
\draw[thick,red,{Circle[scale=0.7]}-{Circle[scale=0.7]}] ($ (\x,2) + (-.15,.35)$) -- ($ (\x,2) + (.15,-.35)$);
}
\foreach \x in {0,2.5}{
\foreach \y in {0,.5,...,2.5}{
\draw[thick,red,{Circle[scale=0.7]}-{Circle[scale=0.7]}] ($ (\x,\y) + (-.15,.35)$) -- ($ (\x,\y) + (.15,-.35)$);
}}
\foreach \y in {.5,1,1.5,2}{
\foreach \x in {.5,2}{
\draw[thick,blue,-{Circle[scale=0.7]}] (\x,\y) -- ($ (\x,\y) + (.15,-.35)$);
}}
\foreach \y in {1,1.5}{
\foreach \x in {1,1.5}{
\draw[thick,blue,-{Circle[scale=0.7]}] (\x,\y) -- ($ (\x,\y) + (.15,-.35)$);
}}
\fill[fill=gray!60, opacity=.8] (.25,.25) -- (.75,.25) -- (.75,.75) -- (1.75,.75) -- (1.75,.25) -- (2.25,.25) -- (2.25,2.25) -- (1.75,2.25) -- (1.75,1.75) -- (.75,1.75) -- (.75,2.25) -- (.25,2.25) -- (.25,.25);
\draw[thick, blue] (.25,.25) -- (.75,.25) -- (.75,.75) -- (1.75,.75) -- (1.75,.25) -- (2.25,.25) -- (2.25,2.25) -- (1.75,2.25) -- (1.75,1.75) -- (.75,1.75) -- (.75,2.25) -- (.25,2.25) -- (.25,.25);
\foreach \y in {.5,1,1.5,2}{
\foreach \x in {.5,2}{
\draw[thick,blue,-{Circle[scale=0.7]}] (\x,\y) -- ($ (\x,\y) + (-.15,.35)$);
}}
\foreach \y in {1,1.5}{
\foreach \x in {1,1.5}{
\draw[thick,blue,-{Circle[scale=0.7]}] (\x,\y) -- ($ (\x,\y) + (-.15,.35)$);
}}
\node[blue] at (1.25,1.25) {$f$};
\foreach \x in {.5,1,1.5,2}{
\filldraw[knot,white] ($ (\x,0) + (-.15,.35)$) circle (.05cm);
\draw[thick,red,knot,-{Circle[scale=0.75]}] ($ (\x,0) + (-.15,.35)$) -- ($ (\x,0) + (.15,-.35)$);
\filldraw[red] ($ (\x,0) + (-.15,.35)$) circle (.05cm);
}
\foreach \x in {1,1.5}{
\filldraw[knot,white] ($ (\x,.5) + (-.15,.35)$) circle (.05cm);
\draw[thick,red,knot,-{Circle[scale=0.75]}] ($ (\x,.5) + (-.15,.35)$) -- ($ (\x,.5) + (.15,-.35)$);
\filldraw[red] ($ (\x,.5) + (-.15,.35)$) circle (.05cm);
}
\node at (1.25,-.65) {$\psi_R(f)= \psi_S(f)$};
}
\end{equation}
One then shows that $L^2(\fA, \psi)\cong \cK_{1_{\cB}}$. 
For each cone $\Lambda$, the corresponding cone von Neumann algebra is defined to be $A_\Lambda\coloneq \pi_\psi(\fA(\Lambda))''$, where $\pi_\psi$ is the GNS construction associated to $\psi$. The family $\cA\coloneq (A_\Lambda)$ is a net of algebras in $B(\cK_{1_{\cB}})$ associated to the poset of (modified) cones:
\begin{itemize}
    \item (isotone) The net clearly satisfies isotony,
    \item (Haag duality) This is \cite[Theorem~4.7]{2506.19969}, and
    \item (absorbing) Also by \cite[Theorem~4.7]{2506.19969}, each $A_\Lambda$ is a direct sum of type I$_\infty$ factors, and is therefore properly infinite.
\end{itemize}
By construction, each Hilbert space $\cK_x$ is a superselection sector for $\cA$ induced by post-composition on finite regions.
\par Now, if $\langle X \rangle$ is the full braided fusion subcategory of $\cB$ generated by $X$, each object of $\mathsf{Hilb}\langle X \rangle$ may be expressed as 
$$\bigoplus_{x \in \text{Irr}\langle X \rangle} H_x \otimes x$$
for some (separable) Hilbert space $H_x$. It is proven in \cite[Theorem~4.8]{2506.19969} that the functor
\begin{equation}
\label{eq:EquivalenceForBraidedNet}
\mathsf{Hilb}\langle X \rangle^{\text{op}} \ni\bigoplus _{x \in \text{Irr}\langle X \rangle} H_x \otimes x \mapsto \bigoplus_{x \in \text{Irr}\langle X \rangle}H_x \otimes \cK_x\in \SSS
\end{equation}
is a unitary equivalence of W$^*$-categories (where each $H_x$ acts as a multiplicity space).

\section{Superselection Sectors via Connes Fusion}\label{sec:connesfusion}

In this section, we construct the Connes fusion of two superselection sectors. The construction of the underlying Hilbert space is analogous to \cite{MR4239831}, but since we are not necessarily working with conformal nets, we use a different approach to defining the representations of each algebra on the fusion of Hilbert spaces.
\subsection{Monoidal product via Connes fusion}

\begin{definition}[Connes fusion]\label{def:Connes-fusion}
    Let $(K,\pi),(L,\sigma)\in \SSS$ and let $p,q\in{\cP}$ be disjoint. Define the \emph{Connes fusion of $K$ and $L$ over $p,q$}, denoted $K(p)\boxtimes L(q)$, to be the completion of $\Hom_{A_{p'}}(H,K)\otimes H\otimes \Hom_{A_{q'}}(H,L)$, with respect to the positive semi-definite sesquilinear form
    \[
    \langle f_1\otimes\xi_1\otimes g_1|f_2\otimes\xi_2\otimes g_2\rangle=\langle g_2^*g_1f_2^*f_1\xi_1|\xi_2\rangle=\langle f_2^*f_1g_2^*g_1\xi_1|\xi_2\rangle.
    \]
    Define the \emph{Connes fusion of $K$ and $L$ over $p$ on the left}, denoted $K(p)\boxtimes L$, to be the completion of $\Hom_{A_{p'}}(H,K)\otimes_{A_p} L$ with respect to the positive semi-definite sesquilinear form
    \[
    \langle f_1\otimes\eta_1|f_2\otimes\eta_2\rangle=\langle \sigma_p(f_2^*f_1)\eta_1|\eta_2\rangle.
    \]
    Define the \emph{Connes fusion of $K$ and $L$ over $q$ on the right}, denoted $K\boxtimes L(q)$, to be the completion of $K\otimes_{A_q}\Hom_{A_{q'}}(H,L)$ with respect to the positive semi-definite sesquilinear form
    \[
    \langle \eta_1\otimes g_1|\eta_2\otimes g_2\rangle=\langle \pi_q(g_2^*g_1)\eta_1|\eta_2\rangle.
    \]
\end{definition}

The next lemma gives examples of dense subsets in the Connes fusion, which will be useful for constructing maps out of the Connes fusion.
\begin{lemma}\label{lemma:dense-subsets-of-Connes-fusion} Let $X\subset \Hom_{A_{p'}}(H,K)$ and $Y\subset L$ be subspaces. If every element of $\Hom_{A_{p'}}(H,K)$ is the limit of a norm-bounded net in $X$ under the strong operator topology, and $Y$ is norm dense, then $X\otimes Y$ is norm dense in $K(p)\boxtimes L$. Similar statements hold for the right and two-sided Connes fusions.
\end{lemma}

\begin{proof}
    Let $\eta\in L$ and suppose $(f_n)\subset X$ is a norm bounded net that converges to $0$ strongly. Then $(f_n^*f_n)^{1/2}\to 0$ strongly. By norm-boundedness, $f_n^*f_n\to 0$ strongly as well as ultrastrongly, hence ultraweakly too. So, $\sigma_p(f_n^*f_n)\to0$ ultraweakly, hence weakly. Thus, $f_n\otimes\eta\to0$, so the map $-\otimes\eta:\Hom_{A_{p'}}(H,K)\to K(p)\boxtimes L$ sends norm-bounded strongly convergent nets to norm convergent nets. Moreover, for all $f\in \Hom_{A_{p'}}(H,K)$, the map $f\otimes-:L\to K(p)\boxtimes L$ is continuous, so the result follows.
\end{proof}

The following lemma proves that both one-sided Connes fusion products are isomorphic as Hilbert spaces to the two-sided Connes fusion 

\begin{lemma}\label{lemma:canonical-evaluation-is-unitary}
    Suppose $p,q\in{\cP}$ are disjoint. Then the canonical evaluation maps
    \begin{align*}
        \cR _{p,q}:K(p)\boxtimes L(q)&\to K(p)\boxtimes L 
        & &\text{and}& 
        \cL_{p,q}:K(p)\boxtimes L(q) &\to K\boxtimes L(q)
        \\
        f\otimes\xi\otimes g &\mapsto f\otimes g(\xi)
        &&&
        f\otimes\xi\otimes g&\mapsto f(\xi)\otimes g 
    \end{align*}
    are unitaries, where we extend $\cR_{p, q}$ and $\cL_{p, q}$ continuously to their whole domains.
\end{lemma}

\begin{proof}
    We have that
    \[
    \langle f_1\otimes g_1(\xi_1)|f_2\otimes g_2(\xi_2)\rangle=\langle \sigma_p(f_2^*f_1)g_1(\xi_1)|g_2(\xi_2)\rangle=\langle f_2^*f_1g_2g_1(\xi_1)|\xi_2\rangle=\langle f_1\otimes \xi_1\otimes g_1|f_2\otimes\xi_2\otimes g_2\rangle,
    \]
    so $\cR_{p,q}$ is isometric. Moreover, $\cR_{p,q}$ is surjective, since there exist unitaries in $\Hom_{A_{q'}}(H,L)$ by the absorbing property. Similarly, $\cL_{p,q}$ is unitary.
\end{proof}

We define maps between the Connes fusion spaces using zig-zags in $\cP$ by first understanding how the different Hilbert spaces relate to each other under inclusions in $\cP$.

\begin{lemma}\label{lemma:inclusions-induce-unitaries}
    The inclusion maps for one-sided and two-sided Connes fusion induce unitary maps.
    
\end{lemma}

\begin{proof}
    Let $p_1,p_2\in{\cP}$ with $p_1\leq p_2$. The diagram below commutes, and the canonical evaluation maps are unitaries by Lemma \ref{lemma:canonical-evaluation-is-unitary}, so it follows that all of the inclusion maps depicted are unitaries. 
    \begin{center}
    \begin{tikzcd}[column sep=0.6em, row sep=0.6em]
	& {K(p_1)\boxtimes L} && {K(p_2)\boxtimes L} & \\
	{} \\
	{K(p_1)\boxtimes L(p_1')} && {K(p_1)\boxtimes L(p_2')} && {K(p_2)\boxtimes L(p_2')} \\
	\\
	& {K\boxtimes L(p_1')} && {K\boxtimes L(p_2')}
	\arrow[from=1-2, to=1-4]
	\arrow["{{{\cR_{p_1,p_1'}}}}", from=3-1, to=1-2]
	\arrow["{{{\cL_{p_1,p_1'}}}}"', from=3-1, to=5-2]
	\arrow["{{{\cR_{p_1,p_2'}}}}"', from=3-3, to=1-2]
	\arrow[from=3-3, to=3-1]
	\arrow[from=3-3, to=3-5]
	\arrow["{{{\cL_{p_1,p_2'}}}}"', from=3-3, to=5-4]
	\arrow["{{{\cR_{p_2,p_2'}}}}"', from=3-5, to=1-4]
	\arrow["{{{\cL_{p_2,p_2'}}}}", from=3-5, to=5-4]
	\arrow[from=5-4, to=5-2]
\end{tikzcd}
\end{center}
\end{proof}

\begin{definition}\label{def:zigzagmaps}
    Suppose $p,q\in{\cP}$ and $\alpha$ is a zig-zag from $p$ to $q$. By Lemma \ref{lemma:inclusions-induce-unitaries}, $\alpha$ induces a unitary map $\alpha_*:K(p)\boxtimes L\to K(q)\boxtimes L$. By an abuse of notation, we will also use $\alpha_*$ to label the induced map $K\boxtimes L(p)\to K\boxtimes L(q)$.
\end{definition}

\begin{remark}\label{remark:natural-actions}
    For $(K,\pi),(L,\sigma)\in\SSS$ and disjoint $p,q\in{\cP}$, there are natural (left) actions of $A_p$ and $A_q$ on $K(p)\boxtimes L(q)$ given as follows: for $x\in A_p$, $y\in A_q$, and $f\otimes\xi\otimes g\in K(p)\boxtimes L(q)$,
    \[
    x\cdot(f\otimes\xi\otimes g)=\pi_p(x)f\otimes\xi\otimes g,\;\;\;\;y\cdot(f\otimes\xi\otimes g)=f\otimes\xi\otimes \sigma_q(y)g.
    \]
    Moreover, there are natural actions of $x\in A_{p}$ and $y\in A_{p'}$ on $f\otimes\eta\in K(p)\boxtimes L$ given by
    \[
    x\cdot(f\otimes\eta)=\pi_p(x)f\otimes\eta,\;\;\;\; y\cdot(f\otimes\eta)=f\otimes\sigma_{p'}(y)\eta,
    \]
    and natural actions of $x\in A_{q'}$ and $y\in A_{q}$ on $\eta\otimes g\in K\boxtimes L(q)$ given by
    \[
    x\cdot(\eta\otimes g)=\pi_{q'}(x)\eta\otimes g,\;\;\;\; y\cdot(\eta\otimes g)=\eta\otimes\sigma_{q}(y)g.
    \]
    In addition, these actions restrict to actions of $A_r$ for any $r\leq p$ or $q$. 
\end{remark}
One can also define right actions of $A_p$ and $A_q$, which satisfy the relations in the fact below.

\begin{fact}[Relative tensor product relations in Connes fusion]\label{fact:relative-tensor-relations}
    For any $f\otimes\xi\otimes g\in K(p)\boxtimes L(q)$, $x\in A_p$, and $y\in A_q$,
    \[
    fx\otimes \xi\otimes g=f\otimes x\xi\otimes g
    \qquad\qquad
    \text{and}
    \qquad\qquad
    f\otimes\xi\otimes gy=f\otimes y\xi\otimes g.
    \]
    For any $f\otimes\eta\in K(p)\boxtimes L$ and $x\in A_p$,
    \[
    fx\otimes\eta=f\otimes \sigma_p(x)\eta.
    \]
    For any $\eta\otimes g\in K\boxtimes L(q)$ and $y\in A_q$,
    \[
    \eta\otimes gy=\pi_q(y)\eta\otimes g.
    \]
\end{fact}
In fact, the algebraic relations in Fact \ref{fact:relative-tensor-relations} span the nullspace of the corresponding sesquilinear form (indeed, our construction in Definition \ref{def:Connes-fusion} can be viewed as the interior tensor product of Hilbert $\rmC^*$-modules c.f.~\cite[Prop.~4.5]{MR1325694}). 
Hence, the forms of Connes fusion in Definition \ref{def:Connes-fusion} can be equivalently defined as completions of relative tensor product spaces under the same (now positive-definite) sesquilinear forms.

\subsection{Action of the net of algebras on the Connes fusion}

By the constructions in the previous subsection, we essentially have a net of one-sided Connes fusion products $K(p)\boxtimes L$ indexed by $p\in\cP$, each equipped with actions of $A_p$ and $A_{p'}$. To define the action of $A_q$ on $K(p)\boxtimes L$ for arbitrary $p,q\in\cP$, the key idea is to use zig-zags via Definition~\ref{def:zigzagmaps} to transport to a space where the action of $A_q$ is already defined. The rest of this subsection is devoted to proving that any choice of zig-zag produces the same action. The next lemma supplies us with the existence of ``short" zig-zags to transport along.

\begin{lemma}\label{lemma:saturation-diagonal-pair}
    The axiom \ref{geom:saturate} implies that for any $p,q\in{\cP}$, at least one of the following holds:
    \begin{enumerate}
        \item There exist $a,b\in{\cP}$ such that $p,q\leq a$ and $b\leq p,q$.
        \item There exist $c,d\in{\cP}$ such that $p,q'\leq c$ and $d\leq p,q'$.
    \end{enumerate}
\end{lemma}

\begin{proof}
By \ref{geom:saturate}, for each adjacent pair in the square below, one of the sets in the pair must be nonempty.
\[\begin{tikzcd}[column sep=.5em, row sep=.5em]
	{p\Cap q} && {p\Cap q'} \\
	\\
	{p'\Cap q} && {p'\Cap q'}
	\arrow[no head, from=1-1, to=1-3]
	\arrow[no head, from=1-3, to=3-3]
	\arrow[no head, from=3-1, to=1-1]
	\arrow[no head, from=3-3, to=3-1]
\end{tikzcd}\]    
It follows that $p\Cap q\neq\emptyset\neq p'\Cap q'$ or $p\Cap q'\neq\emptyset\neq p'\Cap q$. Thus, we may take $b\in p\Cap q$ and $a'\in p'\Cap q'$ in the first case, and $d\in p\Cap q'$ and $c'\in p'\Cap q$ in the second case.
\end{proof}

\noindent The above lemma gives four types of short zig-zags (two each for cases (1) and (2)), namely:
\begin{itemize}
    \item $\alpha \colon p \leq a \geq q$,
    \item $\beta\colon  p\geq b \leq q$,
    \item $\gamma\colon  p\leq c \geq q'$, and
    \item $\delta\colon p \geq d \leq q'$.
\end{itemize}
We will prove in Proposition \ref{prop:action-from-a-b-c-d} that all four zig-zags yield identical actions on the Connes fusion of superselection sectors. Observe that the inclusion maps induced by $a \geq q$ and $d \leq q'$ in the $\alpha$ and $\delta$ zig-zags, respectively, intertwine the actions of $A_q$. Thus, transporting the action of $A_q$ along $\alpha$ and $\delta$ is the same as transporting along the zig-zags $p\leq a$ and $p\geq d$ respectively. We first show that transporting along $\alpha$ or $\delta$ defines a canonical action: once this is done, it will be easier to prove that $\beta$ and $\gamma$ give the same action.

\begin{proposition}\label{prop:action-from-a-and-d}
    Let $(K,\pi),(L,\sigma)\in\SSS$ and $p\in {\cP}$. There exists a unique action $\pi(p)\boxtimes \sigma$ of $\cA$ on $K(p)\boxtimes L$ such that for all $q\in{\cP}$, $x\in A_q$, and $a,d\in{\cP}$ such that $p,q\leq a$ and $d\leq p,q'$, we have that
    \[
    (\pi(p)\boxtimes \sigma)(x)=\alpha_*^{-1}x\alpha_*=\delta_*^{-1}x\delta_*,
    \]
    where $\alpha$ and $\delta$ are the zig-zags $p\leq a\geq q$ and $p\geq d\leq q'$ respectively. Moreover, this action extends the action described in Remark \ref{remark:natural-actions}.
\end{proposition}

\begin{proof}
    Let $x\in A_q$. First suppose there exist $a,d\in{\cP}$ such that $p,q\leq a$ and $d\leq p,q'$. 
    
\[\begin{tikzcd}[column sep=.5em, row sep=.5em]
	& {K(a)\boxtimes L} &&& \\
	\\
	{K(q)\boxtimes L} && {K(p)\boxtimes L} && {K(q')\boxtimes L} \\
	\\
	&&& {K(d)\boxtimes L}
	\arrow["{{\iota_{q,a}}}", from=3-1, to=1-2]
	\arrow["{{\iota_{p,a}}}"', from=3-3, to=1-2]
	\arrow["{{\alpha_*}}", from=3-3, to=3-1]
	\arrow["{{\delta_*}}", from=3-3, to=3-5]
	\arrow["{{\iota_{d,p}}}", from=5-4, to=3-3]
	\arrow["{{\iota_{d,q'}}}"', from=5-4, to=3-5]
\end{tikzcd}\]

    Let $f_j\otimes\eta_j\in \Hom_{A_{d'}}(H,K)\otimes L\subset K(p)\boxtimes L$ for $j=1,2$. Since $\iota_{q,a}$ intertwines $A_q$ and $f_1$ intertwines $A_{d'}$ (and thus $A_q$),
    \begin{align*}
        \langle \alpha_*^{-1}x\alpha_*(f_1\otimes\eta_1)|f_2\otimes\eta_2\rangle
        &=\langle \iota_{p,a}^{-1}x\iota_{p,a}(f_1\otimes\eta_1)|f_2\otimes\eta_2\rangle \\
        &=\langle \pi_a(x)f_1\otimes\eta_1|f_2\otimes\eta_2\rangle \\
        &=\langle \sigma_a(f_2^*\pi_q(x)f_1)\eta_1|\eta_2\rangle \\
        &=\langle \sigma_a(f_2^*f_1x)\eta_1|\eta_2\rangle \\
        &=\langle \sigma_d(f_2^*f_1)\sigma_q(x)\eta_1|\eta_2\rangle.
    \end{align*}
    On the other hand,
    \begin{align*}
        \langle \delta_*^{-1}x\delta_*(f_1\otimes\eta_1)|f_2\otimes\eta_2\rangle
        &=\langle x\delta_*(f_1\otimes\eta_1)|\delta_*(f_2\otimes\eta_2)\rangle \\
        &=\langle x\iota_{d,q'}(f_1\otimes\eta_1)|\iota_{d,q'}(f_2\otimes\eta_2)\rangle\\
        &=\langle f_1\otimes\sigma_q(x)\eta_1|f_2\otimes\eta_2\rangle \\
        &=\langle \sigma_{q'}(f_2^*f_1)\sigma_q(x)\eta_1|\eta_2\rangle \\
        &=\langle \sigma_{d}(f_2^*f_1)\sigma_q(x)\eta_1|\eta_2\rangle.
    \end{align*}
    Since the simple tensors $f\otimes\eta\in \Hom_{A_{d'}}(H,K)\otimes L$ span a dense subset of $K(p)\boxtimes L$ by Lemma \ref{lemma:inclusions-induce-unitaries}, it follows that $\alpha_*^{-1}x\alpha_*=\delta_*^{-1}x\delta_*$.

    Suppose now that $p\Cap q'=\emptyset$. Then by Lemma \ref{lemma:saturation-diagonal-pair}, there exist $a,b\in{\cP}$ such that $p,q\leq a$ and $b\leq p,q$. The diagram below commutes, so $\alpha_*^{-1}x\alpha_*=\beta_*^{-1}x\beta_*$, where $\beta$ is the zig-zag $p\geq b\leq q$.
    \[
\begin{tikzcd}[column sep=.5em, row sep=.5em]
	& {K(a)\boxtimes L} & \\
	\\
	{K(q)\boxtimes L} && {K(p)\boxtimes L} \\
	\\
	& {K(b)\boxtimes L}
	\arrow[from=3-1, to=1-2]
	\arrow[from=3-3, to=1-2]
	\arrow[from=5-2, to=3-1]
	\arrow[from=5-2, to=3-3]
\end{tikzcd}
\]
In particular, any two $a_1,a_2\geq p,q$ define the same action of $x$. Similarly, if $p'\Cap q'=\emptyset$, then there exist $c,d\in{\cP}$ such that $p,q'\leq c$ and $d\leq p,q'$, and one can show that any two $d_1,d_2\leq p,q'$ define the same action of $x$. By Lemma \ref{lemma:saturation-diagonal-pair}, one of the above three cases must occur, so the desired action $\pi(p)\boxtimes \sigma$ exists. That this new action extends the one in Remark \ref{remark:natural-actions} follows from taking $a=p$ for $q=p$, and taking $d=p$ for $q=p'$.
\end{proof}

We next work to show that all four types of short zig-zags from Lemma \ref{lemma:saturation-diagonal-pair} can be used to define the action on the Connes fusion. It is clear that for $p_1,p_2\in{\cP}$ with $p_1\leq p_2$, the canonical inclusion $K(p_1)\boxtimes L\to K(p_2)\boxtimes L$ intertwines the actions of $A_{p_1}$ and $A_{p_2'}$, which are defined on $K(p_1)\boxtimes L$ and $K(p_2)\boxtimes L$ via Remark \ref{remark:natural-actions}. The next lemma is a mild improvement upon this fact, in light of our newly defined action.

\begin{lemma}\label{lemma:inclusions-partially-intertwine}
    Let $p_1,p_2\in{\cP}$ with $p_1\leq p_2$. Then the canonical inclusion $\iota:K(p_1)\boxtimes L\to K(p_2)\boxtimes L$ intertwines the actions of $A_{p_2}$ and $A_{p_1'}$, where the actions are defined as in Proposition \ref{prop:action-from-a-and-d}.
\end{lemma}

\begin{proof}
    Since $p_1\leq p_2\geq p_2$, by Proposition \ref{prop:action-from-a-and-d}, the action of $x\in A_{p_2}$ on $f\otimes\eta\in K(p_1)\boxtimes L$ is given by
    \[
    x\cdot(f\otimes\eta)=\iota^{-1}x\iota(f\otimes\eta),
    \]
    so $\iota$ intertwines $A_{p_2}$. Similarly, since $p_2\geq p_1\leq p_1=p_1''$, the action of $y\in A_{p_1'}$ on $f\otimes\eta\in K(p_2)\boxtimes L$ is given by
    \[
    y\cdot(f\otimes\eta)=\iota y\iota^{-1}(f\otimes\eta),
    \]
    so $\iota$ intertwines $A_{p_1'}$.
\end{proof}

We next observe that ``local'' zig-zags induce the same maps on superselection sectors.

\begin{lemma}\label{lemma:localized-zig-zag}
    Suppose $\zeta=(z_1,y_1,z_2,\dots,z_n,y_n,z_{n+1})$ is a zig-zag in ${\cP}$. If $\zeta$ is contained in some $p\in{\cP}$, then $\zeta_*=\theta_*$, where $\theta$ is the zig-zag $z_1\leq p\geq z_{n+1}$.
\end{lemma}

\begin{proof}
    Immediate from the commutativity of the below diagram.
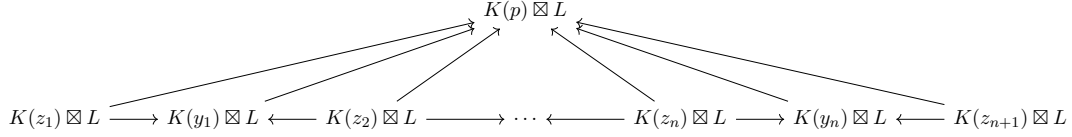
\[
    \adjustbox{scale=0.7}{\begin{tikzcd}
	&&& {K(p)\boxtimes L} &&& \\
	\\
	{K(z_1)\boxtimes L} & {K(y_1)\boxtimes L} & {K(z_2)\boxtimes L} & \cdots & {K(z_n)\boxtimes L} & {K(y_n)\boxtimes L} & {K(z_{n+1})\boxtimes L}
	\arrow[from=3-1, to=1-4]
	\arrow[from=3-1, to=3-2]
	\arrow[from=3-2, to=1-4]
	\arrow[from=3-3, to=1-4]
	\arrow[from=3-3, to=3-2]
	\arrow[from=3-3, to=3-4]
	\arrow[from=3-5, to=1-4]
	\arrow[from=3-5, to=3-4]
	\arrow[from=3-5, to=3-6]
	\arrow[from=3-6, to=1-4]
	\arrow[from=3-7, to=1-4]
	\arrow[from=3-7, to=3-6]
\end{tikzcd}}
\qedhere
\]
\end{proof}
The following proposition justifies that the actions are the same regardless of choice of type of short zig-zag. As a consequence, the actions are furthermore independent of choice of short zig-zag of a particular type. This is also the main tool used to prove that inclusion maps $p_1 \leq p_2$ are intertwiners of superselection sectors.
\begin{proposition}\label{prop:action-from-a-b-c-d}
    Suppose $(K,\pi),(L,\sigma)\in\SSS$ and $p\in {\cP}$. The action $\pi(p)\boxtimes \sigma$ of $A$ on $K(p)\boxtimes L$ from Proposition \ref{prop:action-from-a-and-d} satisfies the following: for all $q\in{\cP}$, $x\in A_q$, and $a,b,c,d\in{\cP}$ such that $p,q\leq a$, $b\leq p,q$, $p,q'\leq c$, and $d\leq p,q'$, we have that
    \[
    (\pi(p)\boxtimes \sigma)(x)=\alpha_*^{-1}x\alpha_*=\beta_*^{-1}x\beta_*=\gamma_*^{-1}x\gamma=\delta_*^{-1}x\delta_*,
    \]
    where $\alpha$, $\beta$, $\gamma$, and $\delta$ are the zig-zags $p\leq a\geq q$, $p\geq b\leq q$, $p\leq c\geq q'$, and $p\geq d\leq q'$ respectively.
\end{proposition}

\begin{proof}
    It remains to show that $(\pi(p)\boxtimes \sigma)(x)=\beta_*^{-1}x\beta_*$ and $(\pi(p)\boxtimes \sigma)(x)=\gamma_*^{-1}x\gamma$ for any $b,c$ as above. First suppose there exists $d\in{\cP}$ such that $d\leq p,q'$. It is immediate that $\gamma_*^{-1}x\gamma_*=\delta_*^{-1}x\delta_*=(\pi(p)\boxtimes \sigma)(x)$, since $\delta_*=\gamma_*$ by the commutativity of the diagram below.
    
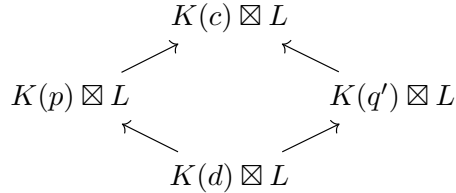
\[
    \begin{tikzcd}[column sep=.5em, row sep=.5em]
	& {K(c)\boxtimes L} & \\
	\\
	{K(p)\boxtimes L} && {K(q')\boxtimes L} \\
	\\
	& {K(d)\boxtimes L}
	\arrow[from=3-1, to=1-2]
	\arrow[from=3-3, to=1-2]
	\arrow[from=5-2, to=3-1]
	\arrow[from=5-2, to=3-3]
\end{tikzcd}
\]
We need to show that for any $b$ as above, $\beta_*^{-1}x\beta_*=\delta_*^{-1}x\delta_*$. This is easy if $p'\Cap q'\neq\emptyset$, but this may not be the case: in general, the situation is the below diagram.
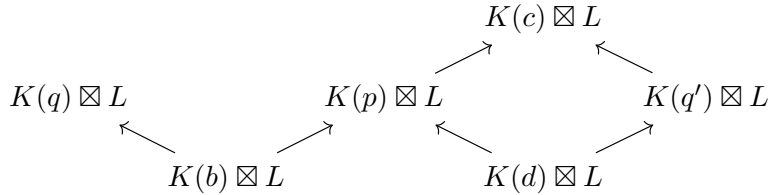
\[
\begin{tikzcd}[column sep=.5em, row sep=.5em]
	&&& {K(c)\boxtimes L} & \\
	\\
	{K(q)\boxtimes L} && {K(p)\boxtimes L} && {K(q')\boxtimes L} \\
	\\
	& {K(b)\boxtimes L} && {K(d)\boxtimes L}
	\arrow[from=3-3, to=1-4]
	\arrow[from=3-5, to=1-4]
	\arrow[from=5-2, to=3-1]
	\arrow[from=5-2, to=3-3]
	\arrow[from=5-4, to=3-3]
	\arrow[from=5-4, to=3-5]
\end{tikzcd}
\]
By Lemma \ref{lemma:inclusions-partially-intertwine}, the canonical inclusions $K(b)\boxtimes L\to K(q)\boxtimes L$ and $K(d)\boxtimes L\to K(q')\boxtimes L$ intertwine $A_q$, so it suffices to show that the composition $\theta:K(b)\boxtimes L\to K(p)\boxtimes L\xleftarrow{}K(d)\boxtimes L$ intertwines $A_q$. By \ref{geom:qSmallqIndicator}, there exists a $q$-small $q$-indicator $\widetilde{d}\leq d\leq p$. Observe that $b\leq p$ is a $q$-small $q$-indicator, so by \ref{geom:ZigZag}, there exists a zig-zag $\zeta=(b=z_1,y_1,z_2,\dots,z_n,y_n,z_{n+1}=\widetilde{d})$ such that each $z_j$ is a $q$-indictor and each $y_j$ is $q$-small. By Lemma \ref{lemma:localized-zig-zag}, $\theta_*$ is equal to the composition of $\zeta_*$ and the canonical inclusion $K(\widetilde{d})\boxtimes L\to K(d)\boxtimes L$. The latter map intertwines $A_q$, so it remains to show that $\zeta_*$ intertwines $A_q$. Suppose $y,z\in{\cP}$ are such that $z$ is a $q$-indicator and $y$ is $q$-small, so that there exists (in particular) $r\in{\cP}$ with $q,y\leq r$. If $z\leq q'$, then the canonical inclusion $K(z)\boxtimes L\to K(y)\boxtimes L$ intertwines $A_{z'}\supset A_{q}$ by Lemma \ref{lemma:inclusions-partially-intertwine}. Otherwise, if $z\leq q$, then the below diagram commutes.
\[
    \begin{tikzcd}[column sep=.5em, row sep=.5em]
	& {K(r)\boxtimes L} & \\
	\\
	{K(q)\boxtimes L} && {K(y)\boxtimes L} \\
	\\
	& {K(z)\boxtimes L}
	\arrow[from=3-1, to=1-2]
	\arrow[from=3-3, to=1-2]
	\arrow[from=5-2, to=3-1]
	\arrow[from=5-2, to=3-3]
\end{tikzcd}
\]
Thus, the map $K(z)\boxtimes L\to K(y)\boxtimes L$ is equal to the map $K(z)\boxtimes L\to K(q)\boxtimes L\to K(r)\boxtimes L\xleftarrow{}K(y)\boxtimes L$, which intertwines $A_q$ by Lemma \ref{lemma:inclusions-partially-intertwine}. Hence, $\beta_*^{-1}x\beta_*=\delta_*^{-1}x\delta_*$. Next we consider the remaining case, namely when there exists $a\in{\cP}$ such that $p,q\leq a$. It is easy to see that $\beta_*^{-1}x\beta_*=\alpha_*^{-1}x\alpha_*$ for any $b$ as in the proposition. To show that $\gamma_*^{-1}x\gamma_*=\alpha_*^{-1}x\alpha_*$, we need to show that the top row in the below diagram intertwines $A_q$.
\[
    \adjustbox{scale=0.55, center}{
    \begin{tikzcd}
	{K(q)\boxtimes L} && {K(a)\boxtimes L} && {K(p)\boxtimes L} && {K(c)\boxtimes L} && {K(q')\boxtimes L} \\
	{K(q)\boxtimes L(q')} & {K(q)\boxtimes L(a')} & {K(a)\boxtimes L(a')} & {K(p)\boxtimes L(a')} & {K(p)\boxtimes L(p')} & {K(p)\boxtimes L(c')} & {K(c)\boxtimes L(c')} & {K(q')\boxtimes L(c')} & {K(q')\boxtimes L(q)} \\
	{K\boxtimes L(q')} && {K\boxtimes L(a')} && {K\boxtimes L(p')} && {K\boxtimes L(c')} && {K\boxtimes L(q)}
	\arrow[from=1-1, to=1-3]
	\arrow[from=1-3, to=2-3]
	\arrow[from=1-5, to=1-3]
	\arrow[from=1-5, to=1-7]
	\arrow[from=1-9, to=1-7]
	\arrow[from=2-1, to=1-1]
	\arrow[from=2-1, to=3-1]
	\arrow[from=2-2, to=1-1]
	\arrow[from=2-2, to=2-1]
	\arrow[from=2-2, to=2-3]
	\arrow[from=2-2, to=3-3]
	\arrow[from=2-4, to=1-5]
	\arrow[from=2-4, to=2-3]
	\arrow[from=2-4, to=2-5]
	\arrow[from=2-4, to=3-3]
	\arrow[from=2-5, to=1-5]
	\arrow[from=2-5, to=3-5]
	\arrow[from=2-6, to=1-5]
	\arrow[from=2-6, to=2-5]
	\arrow[from=2-6, to=2-7]
	\arrow[from=2-6, to=3-7]
	\arrow[from=2-7, to=1-7]
	\arrow[from=2-7, to=3-7]
	\arrow[from=2-8, to=1-9]
	\arrow[from=2-8, to=2-7]
	\arrow[from=2-8, to=2-9]
	\arrow[from=2-8, to=3-7]
	\arrow[from=2-9, to=1-9]
	\arrow[from=2-9, to=3-9]
	\arrow[from=3-3, to=2-3]
	\arrow[from=3-3, to=3-1]
	\arrow[from=3-3, to=3-5]
	\arrow[from=3-7, to=3-5]
	\arrow[from=3-7, to=3-9]
\end{tikzcd}
}
\]
Note that the bottom row intertwines $A_q$ by the previous case, and both outer columns clearly intertwine $A_q$. Since the diagram commutes, we're done.
\end{proof}

It follows immediately from the next proposition that the action of the net of algebras on the left-sided Connes fusion can be defined by transporting along any zig-zag, not just the short zig-zags from Lemma \ref{lemma:saturation-diagonal-pair}.

\begin{proposition}\label{prop:one-sided-inclusions-intertwine}
    The canonical inclusion maps for left-sided Connes fusion intertwine the action of the net of algebras $\cA$.
\end{proposition}

\begin{proof}
    Let $p_1,p_2,q\in{\cP}$ with $p_1\leq p_2$. By \ref{geom:saturate}, there exists $b\in p_1\Cap q\subseteq p_2\Cap q$ or $d\in p_1\Cap q'\subseteq p_2\Cap q'$. In the former case, by Proposition \ref{prop:action-from-a-b-c-d}, the action of $A_q$ on $K(p_1)\boxtimes L$ and $K(p_2)\boxtimes L$ is realized by transporting along the zig-zags $p_1\geq b\leq q$ and $p_2\geq b\leq q$ respectively. The commutativity of the below diagram implies that the canonical inclusion $K(p_1)\boxtimes L\to K(p_2)\boxtimes L$ intertwines the action of $A_q$.
    \[
    \begin{tikzcd}[column sep=.5em, row sep=.75em]
	{K(q)\boxtimes L} && {K(p_2)\boxtimes L} \\
	&& {K(p_1)\boxtimes L} \\
	& {K(b)\boxtimes L}
	\arrow[from=2-3, to=1-3]
	\arrow[from=3-2, to=1-1]
	\arrow[from=3-2, to=1-3]
	\arrow[from=3-2, to=2-3]
\end{tikzcd}
\]
The latter case follows similarly.
\end{proof}

We can similarly define an action $\pi\boxtimes\sigma(p)$ on $K\boxtimes L(p)$. The two canonical evaluation maps allow us to transport the action on the left-sided or right-sided Connes fusion onto the two-sided Connes fusion. The next proposition shows that these two ways produce the same action.

\begin{definitionproposition}
    Let $p,q\in{\cP}$ be disjoint. For all $r\in{\cP}$ and $x\in A_r$, we have
    \[
    (\cR_{p,q})^{-1}x\cR_{p,q}=(\cL_{p,q})^{-1}x\cL_{p,q}.
    \]
    Define the action of $A_r$ on $K(p)\boxtimes L(q)$ by either of the above formulas.
\end{definitionproposition}

\begin{proof}
    First we consider the case $q=p'$. Suppose there exists $a'\in p'\Cap r'$. Then the below diagram commutes and the rightmost column intertwines $A_r$, so $(\cR_{p,p'})^{-1}x\cR_{p,p'}=(\cL_{p,p'})^{-1}x\cL_{p,p'}$.
    \begin{center}
    \begin{tikzcd}[column sep=1.5em, row sep=1em]
	{K(p)\boxtimes L} && {K(a)\boxtimes L} && {K(r)\boxtimes L} \\
	\\
	{K(p)\boxtimes L(p')} & {K(p)\boxtimes L(a')} & {K(a)\boxtimes L(a')} & {K(r)\boxtimes L(a')} & {K(r)\boxtimes L(r')} \\
	\\
	{K\boxtimes L(p')} && {K\boxtimes L(a')} && {K\boxtimes L(r')}
	\arrow[from=1-1, to=1-3]
	\arrow[from=1-5, to=1-3]
	\arrow["{{\cR_{p,p'}}}", from=3-1, to=1-1]
	\arrow["{{\cL_{p,p'}}}"', from=3-1, to=5-1]
	\arrow[from=3-2, to=1-1]
	\arrow[from=3-2, to=3-1]
	\arrow[from=3-2, to=3-3]
	\arrow[from=3-2, to=5-3]
	\arrow[from=3-3, to=1-3]
	\arrow[from=3-3, to=5-3]
	\arrow[from=3-4, to=1-5]
	\arrow[from=3-4, to=3-3]
	\arrow[from=3-4, to=3-5]
	\arrow[from=3-4, to=5-3]
	\arrow[from=3-5, to=1-5]
	\arrow[from=3-5, to=5-5]
	\arrow[from=5-3, to=5-1]
	\arrow[from=5-3, to=5-5]
\end{tikzcd}
    \end{center}
    If $p'\Cap r'=\emptyset$, then \ref{geom:saturate} allows us to pass to a similar case. Now suppose $q\leq p'$. The diagram below commutes.
    \[
    \begin{tikzcd}[column sep=1em, row sep=1em]
	{K(p)\boxtimes L} && {K(p)\boxtimes L(p')} && {K\boxtimes L(p')} \\
	&& {K(p)\boxtimes L(q)} \\
	&&&& {K\boxtimes L(q)}
	\arrow["{{\cR_{p,p'}}}"', from=1-3, to=1-1]
	\arrow["{{\cL_{p,p'}}}", from=1-3, to=1-5]
	\arrow["{{\cR_{p,q}}}", from=2-3, to=1-1]
	\arrow["\varepsilon"', from=2-3, to=1-3]
	\arrow["{{\cL_{p,q}}}", from=2-3, to=3-5]
	\arrow["\iota"', from=3-5, to=1-5]
\end{tikzcd}
    \]
    By Proposition \ref{prop:one-sided-inclusions-intertwine}, $\iota$ intertwines $A_r$. Thus,
    \[
    (\cR_{p,q})^{-1}x\cR_{p,q}=\varepsilon^{-1}(\cR_{p,p'})^{-1}x\cR_{p,p'}\varepsilon=\varepsilon^{-1}(\cL_{p,p'})^{-1}x\cL_{p,p'}\varepsilon=(\cL_{p,q})^{-1}\iota^{-1}x\iota \cR_{p,q}=(\cL_{p,q})^{-1}x\cL_{p,q} .
    \]
\end{proof}

\begin{corollary}
    The canonical evaluation maps (and thus the inclusion maps for two-sided Connes fusion) intertwine the action of the net of algebras $\cA$.
\end{corollary}

Having defined the actions on the underlying fusion of Hilbert spaces, we can now prove that these representations yield a superselection sector.

\begin{proposition}
    Suppose $(K,\pi),(L,\sigma)\in\SSS$ and $p,q\in {\cP}$ are disjoint. Then the Connes fusion $(K(p)\boxtimes L(q),\pi(p)\boxtimes \sigma(q))$ is a superselection sector.
\end{proposition}

\begin{proof}
    \item[\underline{Isotony:}] For $r\leq p$, $(\pi(p)\boxtimes\sigma(q))_p|_{A_r}=(\pi(p)\boxtimes\sigma(q))_r$ and similarly for $r\leq q$. The result follows by transporting along unitaries induced by mutually disjoint zig-zags.
    
    \item[\underline{Locality:}] Suppose $r,s\in{\cP}$ are disjoint. Take a mutually disjoint zig-zag $\zeta$ from $(p,q)$ to $(r,s)$. Then $(\pi(p)\boxtimes\sigma(q))_r=\Ad(\zeta_*)\circ (\pi(r)\boxtimes\sigma(s))_r$ and $(\pi(p)\boxtimes\sigma(q))_s=\Ad(\zeta_*)\circ (\pi(r)\boxtimes\sigma(s))_s$. Moreover, $[(\pi(r)\boxtimes\sigma(s))_r(A_r),(\pi(r)\boxtimes\sigma(s))_s(A_s)]=0$, so $[(\pi(p)\boxtimes\sigma(q))_r(A_r),(\pi(p)\boxtimes\sigma(q))_s(A_s)]=0$.
    
    \item[\underline{Absorbing:}] It suffices to show that $(\pi(p)\boxtimes\sigma(q))_p$ and $(\pi(p)\boxtimes\sigma(q))_q$ are injective, since the actions of the other algebras are unitary conjugations of such actions. Suppose $x\in \ker (\pi(p)\boxtimes\sigma(q))_p$. Then for all $f\otimes\xi\otimes g\in K(p)\boxtimes L(q)$,
    \begin{align*}
        0&=\langle \pi_p(x)f\otimes\xi\otimes g|\pi_p(x)f\otimes\xi\otimes g\rangle
        =\langle f^* \pi_p(x^*x)fg^*g\xi|\xi\rangle
        =\langle\sigma_p(f^* \pi_p(x^*x)f)g\xi|g\xi\rangle \\
        &=\lVert\sigma_p((f^* \pi_p(x^*x)f)^{1/2})g\xi\rVert^2.
    \end{align*}
    Since there exist unitaries in $\Hom_{A_{q'}}(H,L)$, it follows that $\sigma_p((f^* \pi_p(x^*x)f)^{1/2})=0$. But $\sigma_p$ is injective, so $f^* \pi_p(x^*x)f=0$. Thus, $\pi_p(x)f=0$, but there exist unitaries in $\Hom_{A_{p'}}(H,K)$, so $\pi_p(x)=0$, and hence $x=0$ since $\pi_p$ is injective. A similar proof shows that $(\pi(p)\boxtimes\sigma(q))_q$ is injective.
\end{proof}
\begin{remark}
Since the canonical evaluation maps are unitary intertwiners, it follows that the one-sided Connes fusion of two superselection sectors is also a superselection sector.
\end{remark}

The following lemma describes the horizontal composition of morphisms in $\SSS$.

\begin{lemma}
    The two-sided and one-sided Connes fusions are linear $*$-bifunctors.
\end{lemma}

\begin{proof}
    We give a proof for the one-sided left Connes fusion (the other cases are similar). Given $p\in\cP$, $(K_j,\pi^j),(L_j,\sigma^j)\in\SSS$ for $j=1,2$ and $F\in\Hom_A(K_1,K_2)$, $G\in\Hom_A(L_1,L_2)$, define $F\boxtimes_p G:\Hom_{A_{p'}}(H,K_1)\otimes L_1\to K_2(p)\boxtimes L_2$ by
    \[
    (F\boxtimes_pG)(f\otimes \eta)=F\circ f\otimes_{A_p}\! G(\eta).
    \]
    For $x\in A_p$,
    \[
    (F\boxtimes_pG)(fx\otimes \eta)=Ffx\otimes_{A_p}\! G(\eta)=Ff\otimes_{A_p}\!\sigma^2_p(x)G(\eta)=Ff\otimes_{A_p}\! G(\sigma^1_p(x)\eta)=(F\boxtimes_pG)(f\otimes x\cdot\eta)
    \]
    so $F\boxtimes_p G$ descends to a map on the relative tensor product space $\Hom_{A_{p'}}(H,K_1)\otimes_{A_p}\! L_1$.  
    
    Moreover, $F\boxtimes_p G$ is bounded: for $\zeta=\sum_{i=1}^n f_i\otimes_{A_p}\! \eta_i\in \Hom_{A_{p'}}(H,K_1)\otimes_{A_p}\! L_1$, consider the matrix $X=(f_i^*f_j)\in M_n(B(H))\cong B(H^n)$. Let $\sigma^1_p(X)=(\sigma^1_p(f_i^*f_j))\in M_n(B(L_1))$ be the entrywise application of $\sigma^1_p$ on $X$. Letting $\eta=(\eta_1,\dots,\eta_n)^T$, we get 
    \[
    \lVert\zeta\rVert^2=\sum_{i,j=1}^n\langle f_j\otimes\eta_j|f_i\otimes\eta_i\rangle=\sum_{i,j=1}^n\langle \sigma_p^1(f_i^*f_j)\eta_j|\eta_i\rangle=\langle \sigma_p^1(X)\eta|\eta\rangle.
    \]
    Next, let $Y=(f_i^*F^*Ff_j)\in M_n(B(H))$ and $\chi=(G(\eta_1),\dots,G(\eta_n))^T$. Then $\lVert(F\boxtimes_p G)(\zeta)\rVert^2=\langle\sigma^2_p(Y)\chi|\chi\rangle$. Now, $Y\leq\lVert F\rVert^2X$, so $\sigma^2_p(Y)\leq\lVert F\rVert^2\sigma^2_p(X)$ since $\sigma^2_p$ is a $*$-homomorphism (and hence completely positive). Since $(G^*G)\Id_{L_1^n}$ and $\sigma^1_p(X)$ are two commuting positive operators, $G^*G\sigma^1_p(X)\leq\lVert G\rVert^2\sigma^1_p(X)$, so
    \[
    \lVert(F\boxtimes_p G)(\zeta)\rVert^2=\langle\sigma^2_p(Y)\chi|\chi\rangle\leq \lVert F\rVert^2\langle\sigma^2_p(X)\chi|\chi\rangle=\lVert F\rVert^2\langle G^*G\sigma^1_p(X)\eta|\eta\rangle\leq\lVert F\rVert^2\lVert G\rVert^2\lVert\zeta\rVert^2.
    \]
    Hence, $F\boxtimes_p G$ is bounded, and so extends to a map $K_1(p)\boxtimes L_1\to K_2(p)\boxtimes L_2$. Moreover, $F\boxtimes_p G$ intertwines the actions of $A_p$ and $A_{p'}$, and is natural in $p$: if $p\leq q$, then the diagram below commutes.
    \[\begin{tikzcd}[column sep=1.5em, row sep=.75em]
	{K_1(p)\boxtimes L_1} && {K_2(p)\boxtimes L_2} \\
	\\
	{K_1(q)\boxtimes L_1} && {K_2(q)\boxtimes L_2}
	\arrow["{F\boxtimes_p G}", from=1-1, to=1-3]
	\arrow[from=1-1, to=3-1]
	\arrow[from=1-3, to=3-3]
	\arrow["{F\boxtimes_q G}"', from=3-1, to=3-3]
\end{tikzcd}\]
Thus, $F\boxtimes_p G$ intertwines $A$. It is now easy to verify that Connes fusion is a linear $*$-bifunctor.
\end{proof}

Since we have described the Connes fusion as a bifunctor on $\SSS$, all that remains to prove that $(\SSS, \boxtimes_p)$ is a monoidal category is to provide associators and unitors. 
We will do this for the right-sided Connes fusion, since we will later prove an equivalence between $\mathsf{SSS}_p$ under composition and $\mathsf{SSS}$ under right-sided Connes fusion over $p$. Left-sided Connes fusion gives an equivalent monoidal product (see Remark \ref{remark:left-right-monoidal-equivalence}).

\begin{proposition}\label{prop:sss-monoidal}
    $\SSS$ is a monoidal category under right-sided Connes fusion with associator 
    \begin{align*}
        \alpha_{A, B, C} \colon (A \boxtimes B(p)) \boxtimes C(p) &\to A \boxtimes (B \boxtimes C(p))(p) \\
        (\eta_a \otimes g_b) \otimes g_c &\mapsto \eta_a \otimes (g_B(-) \otimes g_c)
    \end{align*}
    and unitors
\begin{align*}
H \boxtimes K (p) &\overset{\lambda_K}{\longrightarrow} K
&&
\text{and}
&
K \boxtimes H (p) &\overset{\rho_K}{\longrightarrow} K 
\\
\eta \otimes f &\longmapsto {f(\eta)} 
&&& {\xi \otimes g} &\longmapsto {\pi_p(g)\xi}.
\end{align*}
    \end{proposition}

\begin{proof}
    The map $\alpha_{A,B,C}$ is, a priori, only defined upon a dense subspace. It may be extended by continuity to a unitary intertwiner in $\SSS$ as the following diagram commutes, where the bottom horizontal map is reparenthesization.
    
    \[\begin{tikzcd}[column sep=1em, row sep=1.5em]
	{(A \boxtimes B(p)) \boxtimes C(p)} && {A \boxtimes (B \boxtimes C(p))(p)} \\
	\\
	{(A(p') \boxtimes B(p)) \boxtimes C(p)} && {A(p') \boxtimes (B \boxtimes C(p))(p)} \\
	\\
	{(A(p') \boxtimes B) \boxtimes C(p)} && {A(p') \boxtimes(B \boxtimes C(p))}
	\arrow["{\alpha_{A, B, C}}", from=1-1, to=1-3]
	\arrow["{\cL_{p', p} \boxtimes \id_{C(p)}}", from=3-1, to=1-1]
	\arrow["{\cR_{p', p} \boxtimes \id_{C(p)}}"', from=3-1, to=5-1]
	\arrow["{\cL_{p', p}}"', from=3-3, to=1-3]
	\arrow["{\cR_{p', p}}", from=3-3, to=5-3]
	\arrow[from=5-1, to=5-3]
\end{tikzcd}\]
The unitors are clearly unitary intertwiners. 
As the associator and unitors are manifestly natural, the monoidal equivalence in Proposition \ref{prop:monoidal-equivalence-to-localized-endomorphisms} lifts the monoidal coherence of $(\SSS_p, \circ_p, H)$.
(Alternatively, the monoidal coherence is directly proven in Appendix~\ref{app:pentagon-triangle}.)
\end{proof}

\begin{remark}\label{remark:left-right-monoidal-equivalence}
Similarly, SSS is a monoidal category under left-sided Connes fusion $\tensor*[_p]{\boxtimes}{}$. One can show that the tensorator $K(p)\boxtimes L\to K(p)\boxtimes L(p')\to K\boxtimes L(p')$ implements a monoidal equivalence between $(\SSS,\tensor*[_p]{\boxtimes}{},H)$ and $(\SSS,\tensor*[]{\boxtimes}{_{p'}},H)$. Moreover, if $p,q\in\cP$ with $p\leq q$, then there is a monoidal equivalence between $(\SSS,\tensor*[]{\boxtimes}{_p},H)$ and $(\SSS,\tensor*[]{\boxtimes}{_q},H)$, with tensorator given by the canonical inclusion map. Since $\cP$ is connected, it follows that any choices of left-sided or right-sided Connes fusion over any $p\in\cP$ yield equivalent monoidal categories. 
\end{remark}

\subsection{Braiding for Connes fusion}

We define a braiding $\beta$ on $(\SSS,\boxtimes_p,H)$ induced by a choice of splitting $(r,s)$ of $p$. 
Given superselection sectors $(K, \pi)$ and $(L,\sigma)$, let $\beta_{K,L}$ be the composition of the following sequence of unitary intertwiners:
\[
K \boxtimes L(p) \leftarrow K \boxtimes L(s) \rightarrow K \boxtimes L(r') \leftarrow K(r) \boxtimes L(r') \cong L(r') \boxtimes K(r) \rightarrow L \boxtimes K(r) \rightarrow L \boxtimes K(p)
\]
That is, for any $\eta\otimes g\in K\otimes\Hom_{A_{s}'}(H,L) $ (which span a dense subset of $K\boxtimes L(p)$) and unitary $u \in \Hom_{A_r'}(H,K)$,
\[
\beta_{K,L}:\eta\otimes g\mapsto u\otimes u^*(\eta)\otimes g\mapsto g\otimes u^*(\eta)\otimes u\mapsto g(u^*(\eta))\otimes u\in L\boxtimes K(p)
\]
In particular, for any $f \in \Hom_{A_r'}(H,K)$, $g \in \Hom_{A_s'}(H,L)$, $\xi \in H$, and unitary $u \in \Hom_{A_r'}(H,K)$,
\[
\beta_{K,L}(f(\xi)\otimes g)=g(u^*f(\xi))\otimes u=\sigma_r(u^*f)g(\xi)\otimes u=g(\xi)\otimes uu^*f=g(\xi)\otimes f,
\]
so that $\beta_{K,L}\circ(f(-)\otimes g)=g(-)\otimes f$. It is clear that swapping the roles of $r$ and $s$ yields the reverse braiding. Moreover, if $(r_1,s_1)$ and $(r_2,s_2)$ are two splittings of $p$ with $r_1\leq r_2$ and $s_1\leq s_2$, then the two braidings they induce agree on a dense subset. Hence, any two splittings of $p$ that are related by a mutually disjoint zig-zag conatined in $p$ give the same braiding.

To show that $\beta$ is indeed a braiding, one can verify the hexagon identities directly, as is done in Appendix \ref{app:hexagon-identities}. Alternatively, it suffices to show that the monoidal equivalence in Proposition \ref{prop:monoidal-equivalence-to-localized-endomorphisms} is compatible with $\beta$ and the braiding $\gamma$ on $\SSS_p$, which is done in Proposition \ref{prop:braided-equivalence-with-localized-enodmorphisms}.

\begin{remark}
    A braiding can similarly be defined on $(\SSS,\tensor*[_p]{\boxtimes}{},H)$. Using the reflection property \cite[(GA5)]{MR4927814} GA5, one can show there is a braided monoidal equivalence between $(\SSS,\tensor*[_p]{\boxtimes}{},H)$ and $(\SSS,\tensor*[]{\boxtimes}{_{p'}},H)$. Moreover, if $p,q\in\cP$ with $p\leq q$, then there is a braided monoidal equivalence between $(\SSS,\tensor*[]{\boxtimes}{_p},H)$ and $(\SSS,\tensor*[]{\boxtimes}{_q},H)$. Since $\cP$ is connected, it follows that any choices of one-sided left or right Connes fusion over any $p\in\cP$ yields equivalent braided monoidal categories (up to reversal of braiding). 
\end{remark}

\section{Equivalence of Superselection Categories}\label{sec:equivalencewithendomorphisms}
In this section, we prove Theorem \ref{thm:A}: $(\SSS, \boxtimes_p)$ is braided equivalent to the category $(\SSS_p, \circ_p, H)$ constructed in \cite{MR4927814}. We first give an overview of the fusion and braiding for $\SSS_p$, the category of superselection sectors localized at $p$, constructed in \cite{MR4927814}. A superselection sector $(K, \pi)$ is \textit{localized} at $p \in \cP$ if $K=H$ (the vacuum Hilbert space) and $\pi_{p'}=\Id_{A_{p'}}$. We will denote $(H,\pi)\in\SSS_p$ by $H_\pi$ or simply by $\pi$. We use $\mathbbm{1}$ to denote the vacuum representation, i.e., $\mathbbm{1}_q=\Id_{A_q}$ for all $q\in{\cP}$. The absorbing property implies that the inclusion functor $\SSS_p\to \SSS$ is essentially surjective.
\par The fusion of $\pi,\sigma\in\mathsf{SSS}_p$ is constructed in \cite{MR4927814} as follows:
by \ref{geom:qSmallqIndicator}, for any $q\in\cP$, there exists a $q$-small $q$-indicator $\widetilde{p}\leq p$. Picking unitaries $u_{\widetilde{p}}\in\Hom_{A_{\widetilde{p}'}}(H,H_\pi)$ and $v_{\widetilde{p}}\in\Hom_{A_{\widetilde{p}'}}(H,H_\sigma)$, we obtain the superselection sectors $\pi^{\widetilde{p}}(-)=u_{\widetilde{p}}^*\pi(-)u_{\widetilde{p}}$ and $\sigma^{\widetilde{p}}(-)=v_{\widetilde{p}}^*\sigma(-)v_{\widetilde{p}}$ localized at $\widetilde{p}$. We then set
\[
(\pi\circ_p\sigma)_q(x)
        :=u_{\widetilde{p}}\pi^{\widetilde{p}}_p(v_{\widetilde{p}})(\pi^{\widetilde{p}}_q\circ\sigma^{\widetilde{p}}_q)(x)\pi^{\widetilde{p}}_p(v_{\widetilde{p}}^*)u_{\widetilde{p}}^*.
\]
It was proved in \cite{MR4927814} that this definition yields a valid superselection sector localized at $p$ and is independent of the choices of $\widetilde{p}$, $u_{\widetilde{p}}$, and $v_{\widetilde{p}}$. Moreover, it was proven that $(\SSS_p,\circ_p)$ is a strict monoidal category with monoidal unit $\mathbbm{1}$.

For a fixed splitting $(r,s)$ of $p$ and $\pi,\sigma\in\SSS_p$, the braiding is given in \cite{MR4927814} by
\[
\gamma_{\pi,\sigma}:=\sigma_p(u_r)v_su_r^*\pi_p(v_s^*)
\]
for any choice of unitaries $u_r\in\Hom_{A_{r'}}(H,H_\pi)\subseteq A_p$ and $v_s\in\Hom_{A_{s'}}(H,H_\sigma)\subseteq A_p$.
It was proven that $\gamma_{\pi,\sigma}$ is independent of the choice of unitaries, that swapping the roles of $r$ and $s$ produces the reverse braiding $\gamma_{\pi,\sigma}^{-1}$, and that $\gamma_{\pi,\sigma}$ is independent of the choice of splitting, up to this swapping.

We now give a simplified form of the fusion on $(\SSS_p, \circ_p)$.

\begin{fact}
    Let $p,q\in{\cP}$. For all $H_\pi,H_\sigma\in\SSS_p$ and $x\in A_q$, the composition $(\pi\circ_p\sigma)_q(x)$ can be expressed as follows:
    \begin{itemize}
        \item If there exists $a\in{\cP}$ with $a\geq p,q$, then
        $
        (\pi\circ_p\sigma)_q(x)=(\pi_a\circ\sigma_q)(x).
        $
        
        \item If there exists $d\in{\cP}$ with $d\leq p,q'$, then
        $        (\pi\circ_p\sigma)_q(x)=\pi_p(u)\pi_q(x)\pi_p(u^*)
        $
        for any unitary $u\in\Hom_{A_{d'}}(H,H_\sigma)\subseteq A_p$.
    \end{itemize}
    Note that by Lemma \ref{lemma:saturation-diagonal-pair}, at least one of the above scenarios occurs.
\end{fact}

\begin{proof}
    If there is $a\geq p,q$, then by isotony,
    \[
    (\pi\circ_p\sigma)_q(x)=(\pi\circ_p\sigma)_a(x)=\pi_a(\sigma_a(x))=\pi_a(\sigma_q(x)).
    \]
    If there is $d\leq p,q'$, then by \ref{geom:qSmallqIndicator}, we may take a $q$-small $q$-indicator $\widetilde{p}\leq d\leq p,q'$. Taking unitaries $u_{\widetilde{p}}\in\Hom_{A_{\widetilde{p}'}}(H,H_\pi)$ and $v_{\widetilde{p}}\in\Hom_{A_{\widetilde{p}'}}(H,H_\sigma)$, we obtain the superselection sectors $\pi^{\widetilde{p}}(-)=u_{\widetilde{p}}^*\pi(-)u_{\widetilde{p}}$ and $\sigma^{\widetilde{p}}(-)=v_{\widetilde{p}}^*\sigma(-)v_{\widetilde{p}}$ localized at $\widetilde{p}$ (and hence at $q'$). Thus,
    \begin{align*}
        (\pi\circ_p\sigma)_q(x)
        &=u_{\widetilde{p}}\pi^{\widetilde{p}}_p(v_{\widetilde{p}})(\pi^{\widetilde{p}}_q\circ\sigma^{\widetilde{p}}_q)(x)\pi^{\widetilde{p}}_p(v_{\widetilde{p}}^*)u_{\widetilde{p}}^*
        =\pi_p(v_{\widetilde{p}})u_{\widetilde{p}}\pi^{\widetilde{p}}_q(x)u_{\widetilde{p}}^*\pi_p(v_{\widetilde{p}}^*)
        =\pi_p(v_{\widetilde{p}})\pi_q(x)\pi_p(v_{\widetilde{p}}^*)\pi_p(uu^*) \\
        &=\pi_p(v_{\widetilde{p}})\pi_q(x)\pi_d(v_{\widetilde{p}}^*u)\pi_p(u^*)
        =\pi_p(v_{\widetilde{p}})\pi_d(v_{\widetilde{p}}^*u)\pi_q(x)\pi_p(u^*) 
        =\pi_p(u)\pi_q(x)\pi_p(u^*).
        \qedhere
    \end{align*}
\end{proof}

We now work towards proving the equivalence of the monoidal structures on $SSS$ and $SSS_p$. We first prove the Connes fusion of localized sectors is isomorphic to the $\circ_p$ fusion of sectors.

\begin{lemma}
    Suppose $p\in{\cP}$ and $H_\pi,H_\sigma\in\SSS_p$. Define the map $\nu:H_\pi\boxtimes H_\sigma(p)\to H_{\pi\circ_p\sigma}$ by
    \[
    \eta\otimes g=\pi_p(g)\eta\otimes 1\xmapsto{\nu} \pi_p(g)\eta
    \]
    for $\eta\in H_\pi$ and $g\in\Hom_{A_{p'}}(H,H_\sigma)=A_p$. Then $\nu$ is a unitary intertwiner.
\end{lemma}

\begin{proof}
    It is easy to see that the map defined above is isometric and surjective, so it remains to show that $\nu$ intertwines $A_q$ for every $q\in{\cP}$. Let $x\in A_q$.
    \item[\underline{Case 1: There exists $a\in{\cP}$ with $a\geq p,q$.}] Let $\eta\otimes g\in H_\pi\boxtimes H_\sigma(p)$. Note that $\sigma$ is localized at $a$, so $\sigma_a(x)\in A_a$. Thus,
    \begin{align*}
        x\cdot(\eta\otimes_{A_p}\! g)
        &=\iota_{p,a}^{-1}x\cdot\iota_{p,a}(\pi_p(g)\eta\otimes_{A_p}\! 1)
        =\iota_{p,a}^{-1}(\pi_p(g)\eta\otimes_{A_a}\!\sigma_a(x))
        =\iota_{p,a}^{-1}(\pi_a(\sigma_a(x))\pi_p(g)\eta\otimes_{A_a}\!1) \\
        &=\pi_a(\sigma_q(x))\pi_p(g)\eta\otimes_{A_p}\!1,
    \end{align*}
    so $\nu(x\cdot(\eta\otimes_{A_p}\!g))=\pi_a(\sigma_q(x))\pi_p(g)\eta=(\pi\circ_p\sigma)_q(x)\pi_p(g)\eta
        =x\cdot\nu(\eta\otimes_{A_p}\! g)$.

    \item[\underline{Case 2: There exists $d\in{\cP}$ with $d\leq p,q'$.}] It suffices to prove that $\nu$ intertwines the action of $x$ on every $\eta \otimes g \in H_\pi \otimes_{A_p} \Hom_{A_d'}(H, H_\sigma)$, which span a dense subset of $H_\pi \boxtimes_p H_\sigma$. We have
    \[
    \nu(x\cdot(\eta\otimes_{A_p}\! g))=\nu(\iota_{d,p}x\cdot(\eta\otimes_{A_d}\!g))=\nu(\iota_{d,p}(\pi_{d'}(x)\eta\otimes_{A_d}\!g))=\nu(\pi_q(x)\eta\otimes_{A_p}g)=\pi_p(g)\pi_q(x)\eta.
    \]
    On the other hand, for any unitary $u\in\Hom_{A_{d'}}(H,H_\sigma)\subseteq A_p$,
    \begin{align*}
        x\cdot\nu(\eta\otimes_{A_p}\!g)
        &=(\pi\circ_p\sigma)_q(x)\pi_p(g)\eta
        =\pi_p(u)\pi_q(x)\pi_p(u^*)\pi_p(g)\eta 
        =\pi_p(u)\pi_q(x)\pi_d(u^*g)\eta \\
        &=\pi_p(u)\pi_d(u^*g)\pi_q(x)\eta
        =\pi_p(g)\pi_q(x)\eta.
        \qedhere
    \end{align*}
\end{proof}

We can now prove that the two categories are equivalent via the inclusion functor.

\begin{proposition}\label{prop:monoidal-equivalence-to-localized-endomorphisms}
    For all $p\in{\cP}$, the inclusion functor $(\SSS_p,\circ_p,H)\to (\SSS,\boxtimes_p,H)$ is an equivalence of monoidal categories.
\end{proposition}

\begin{proof}
    For $(H,\pi),(H,\sigma),(H,\tau)\in\SSS_p$, we need to verify the commutativity of the below diagram.
    \[
    \begin{tikzcd}[column sep=1em, row sep=1em]
	{(H_\pi\boxtimes H_\sigma(p))\boxtimes H_\tau(p)} && {H_\pi\boxtimes (H_\sigma\boxtimes H_\tau(p))(p)} \\
	{H_{\pi\circ_p\sigma}\boxtimes H_\tau(p)} && {H_\pi\boxtimes H_{\sigma\circ_p\tau}(p)} \\
	{H_{(\pi\circ_p\sigma)\circ_p\tau}} && {H_{\pi\circ_p(\sigma\circ_p\tau)}}
	\arrow[from=1-1, to=1-3]
	\arrow[from=1-1, to=2-1]
	\arrow[from=1-3, to=2-3]
	\arrow[from=2-1, to=3-1]
	\arrow[from=2-3, to=3-3]
	\arrow[from=3-1, to=3-3]
\end{tikzcd}
    \]
    Note that the bottom map is the identity. For $(\eta\otimes f)\otimes g\in (H_\pi\boxtimes H_\sigma(p))\boxtimes H_\tau(p)$, the bottom left path gives
    \[
    (\eta\otimes f)\otimes g\mapsto \pi_p(f)\eta\otimes g\mapsto (\pi\circ_p\sigma)_p(g)\pi_p(f)\eta=\pi_p(\sigma_p(g)f)\eta,
    \]
    while the top right path gives
    \[
    (\eta\otimes f)\otimes g\mapsto\eta\otimes(f(-)\otimes g)\mapsto\eta\otimes(\nu\circ(f(-)\otimes g))=\eta\otimes\sigma_p(g)f\mapsto\pi_p(\sigma_p(g)f)\eta.
    \]
    Since $\SSS_p$ is strictly associative and has the same monoidal unit as $\SSS$, the remaining coherence conditions are $\rho_\pi=\nu_{\pi,\mathbbm{1}}$ and $\lambda_\pi=\nu_{\mathbbm{1},\pi}$, which are easily verified.
\end{proof}

Finally, we prove the equivalence of the braidings on $SSS$ and $SSS_p$. 

\begin{proposition}\label{prop:braided-equivalence-with-localized-enodmorphisms}
    Fix a splitting $(r,s)$ of $p\in\cP$. For all $(H,\pi),(H,\sigma)\in\SSS_p$, the diagram below commutes, where $\gamma$ and the candidate braiding $\beta$ are both induced by $(r,s)$.
    \[\begin{tikzcd}[column sep=1em, row sep=1em]
	{H_\pi\boxtimes H_\sigma(p)} && {H_\sigma\boxtimes H_\pi(p)} \\
	\\
	{H_{\pi\circ_p\sigma}} && {H_{\sigma\circ_p\pi}}
	\arrow["{\beta_{\pi,\sigma}}", from=1-1, to=1-3]
	\arrow["{\nu_{\pi,\sigma}}"', from=1-1, to=3-1]
	\arrow["{\nu_{\sigma,\pi}}", from=1-3, to=3-3]
	\arrow["{\gamma_{\pi,\sigma}}", from=3-1, to=3-3]
\end{tikzcd}\]
Hence, $\beta$ is a valid braiding on SSS, and the monoidal equivalence from Proposition \ref{prop:monoidal-equivalence-to-localized-endomorphisms} is a braided equivalence.
\end{proposition}

\begin{proof}
    By denseness, it suffices to prove commutativity on $\eta\otimes g\in H_\pi \otimes_{A_p} \Hom_{A_s'}(H, H_\sigma)\subset H_\pi\boxtimes H_\sigma(p)$. Take unitaries $u_r\in\Hom_{A_{r'}}(H,H_\pi)\subseteq A_p$ and $v_s\in\Hom_{A_{s'}}(H,H_\sigma)\subseteq A_p$, and define $\pi^r\in\SSS_r$ and $\sigma^s\in\SSS_s$ by $\pi^r(-)=u_r^*\pi(-)u_r$ and $\sigma^s(-)=v_s^*\sigma(-)v_s$. Then the bottom left path gives
\begin{align*}
    \eta\otimes g\mapsto \pi_p(g)\eta\mapsto \sigma_p(u_r)v_su_r^*\pi_p(v_s^*)\pi_p(g)\eta
    =\sigma_p(u_r)v_s\pi^r_p(v_s^*)\pi^r_p(g)u_r^*\eta
    &=\sigma_p(u_r)v_s\pi^r_s(v_s^*g)u_r^*\eta \\
    &=\sigma_p(u_r)gu_r^*\eta,
\end{align*}
while the top right path gives $\eta\otimes g\mapsto g(u_r^*(\eta))\otimes u_r\mapsto \sigma_p(u_r)gu_r^*\eta$.
\end{proof}
    
\section{Application to Unitary Braided Categorical Nets}\label{sec:braidedcatnets}
In this section, we will prove that the category of superselection sectors of a braided categorical net (constructed in \cite{2506.19969}) is braided equivalent to the unitary $\mathsf{ind}$-completion of the (opposite of the) original braided category. 
We prove the slightly more general statement where the braided category is replaced by the full subcategory generated by an object. To fix some notation, let $R_n$ denote the square centered at the origin of radius $n$, and let $\Lambda_n\coloneq \Lambda \cap R_n$ for a cone $\Lambda$.
Now, let $\cB$ be a unitary braided fusion category and fix an object $X\in\cB$ such that $1_{\cB}$ is a subobject of $X$. 
Take $N\in\mathbb{N}$ such that every simple object in $\langle X\rangle$ is isomorphic to a subobject of $X^{\otimes N}$. 
For the rest of this section, we will consider $n\in\mathbb{N}$ large enough so that $|R_n|\geq N$.

We now define a map $\phi_n:\Hom_{\mathfrak{A}(R_n\setminus \Lambda)}(K_1^{R_n},K_x^{R_n})\otimes_{\mathfrak{A}(\Lambda_n)}K_y^{R_n}\to K_{x\otimes y}^{R_n}$ which will serve as our tensorator on local regions. 
An application of the Yoneda lemma to semisimple categories yields a more explicit description of the space $\Hom_{\mathfrak{A}(R_n\setminus \Lambda)}(K_1^{R_n},K_x^{R_n})$.

\begin{lemma}\label{lemma:yoneda}
    Let $\cC$ be a finitely semisimple category and suppose $Y\in\cC$ satisfies $\cC(c\to Y)\neq0$ for all $c\in\Irr{(\cC)}$. Then for any $a,b\in\cC$,
    \[
    \Hom_{\End(Y)}(\cC(a, Y),\cC(b, Y))\cong \cC(b, a).
    \]
    That is, the maps from $\cC(a, Y)$ to $\cC(b, Y)$ that commute with the left action of $\End(Y)$ by post-composition are given exactly by pre-composition by elements of $\cC(b, a)$.
\end{lemma}
\begin{proof}
By considering the direct sum $a\oplus b$, this is immediate from \cite[Semisimple Commutation Lemma 3.8]{2506.19969}.
Here is another direct proof for the convenience of the reader.
    Since $Y$ contains a copy of every simple object, every $f\in \Hom_{\End(Y)}(\cC(a, Y),\cC(b, Y))$ extends uniquely to a natural transformation $\eta:\cC(a,-)\Rightarrow\cC(b,-)$ with $\eta_Y=f$, and all natural transformations arise this way. 
    The result then follows from the Yoneda Lemma.
\end{proof}

\begin{corollary}\label{cor:rewrite-hom-space}
    $\Hom_{\mathfrak{A}(R_n\setminus \Lambda)}(K_1^{R_n},K_x^{R_n})\cong \cB(x\otimes X^{\otimes R_n}, X^{\otimes R_n})$ as $\mathfrak{A}(R_n)$-bimodules.
\end{corollary}
\begin{proof}
    By Lemma \ref{lemma:yoneda},
    \begin{align*}
        \Hom_{\mathfrak{A}(R_n\setminus \Lambda)}(K_1^{R_n},K_x^{R_n})
        &\cong \Hom_{\mathfrak{A}(R_n\setminus \Lambda)}(\cB(\overline{X}^{\otimes R_n}, X^{\otimes R_n\setminus\Lambda}),\cB(\overline{X}^{\otimes R_n}\otimes x, X^{\otimes R_n\setminus\Lambda})) \\
        &\cong \cB(\overline{X}^{\otimes R_n}\otimes x, \overline{X}^{\otimes R_n}) \\
        &\cong \cB(x\otimes X^{\otimes R_n}, X^{\otimes R_n}).
        \qedhere
    \end{align*}
\end{proof}
Diagrammatically, the isomorphism in Corollary \ref{cor:rewrite-hom-space} is given as follows: every $f\in \cB(x\otimes X^{\otimes R_n}, X^{\otimes R_n})$ corresponds to the map below.

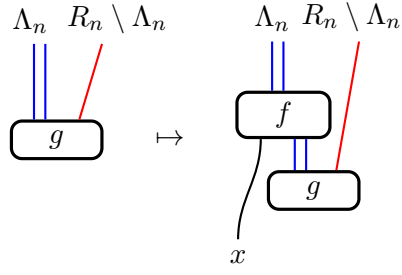
\begin{center}
\begin{tikzpicture}[baseline={(current bounding box.center)},thick]
\draw(0,0) node[draw=black,fill=white,very thick,rounded corners,baseline=center, inner xsep=0.5cm] (a){$g$};
\draw [color = blue]([xshift=-0.15cm]a.north) --++ (0, 1);
\draw [color = blue]([xshift=-0.3cm]a.north) --++ (0, 1);
\draw [color=red] ([xshift=0.3cm]a.north) --++ (0.3, 1);
\node[above] at ([xshift=-0.35cm, yshift=1cm]a.north) {$\Lambda_n$};
\node[above] at ([xshift=0.8cm, yshift=1cm]a.north) {$R_n \setminus \Lambda_n$};

\node[] (b) at ([xshift=1.5cm]a) {$\mapsto$};

\draw([xshift=1.5cm, yshift=0.33cm]b) node[draw=black,fill=white,very thick,rounded corners,baseline=center, inner xsep=0.5cm] (c){$f$};
\draw [color = blue]([xshift=0cm]c.north) --++ (0, 0.65) coordinate (d);
\draw [color = blue]([xshift=-0.15cm]c.north) to ([xshift=-0.15cm]d);
\draw [color = blue]([xshift=0.15cm]c.south) --++ (0, -0.5);
\draw [color = blue]([xshift=0.3cm]c.south) --++ (0, -0.5);
\node[above] at ([xshift=-0.15cm]d) {$\Lambda_n$};

\draw([xshift=0.4cm, yshift=-1cm]c) node[draw=black,fill=white,very thick,rounded corners,baseline=center, inner xsep=0.5cm] (a){$g$};

\draw ([xshift=-0.3cm]c.south) to [out=-90, in=90] node[pos=1, below] {$x$} ([xshift=-0.6cm, yshift=-1.325cm]c.south);
\draw [color=red] ([xshift=0.3cm]a.north) to ([xshift=1cm]d) coordinate (e);
\node[above] at ([xshift=-0.1cm]e) {$R_n \setminus \Lambda_n$};

\end{tikzpicture}
\end{center}

\noindent We can now give a graphical description of $\phi_n$.
\begin{lemma}
    The map $\phi_n:\Hom_{\mathfrak{A}(R_n\setminus \Lambda)}(K_1^{R_n},K_x^{R_n})\otimes_{\mathfrak{A}(R_n)}K_y^{R_n}\to K_{x\otimes y}^{R_n}$ given (under the isomorphism in Corollary \ref{cor:rewrite-hom-space}) by
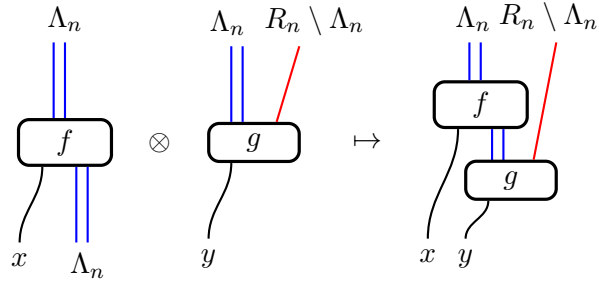
    \begin{center}
        \begin{tikzpicture}[baseline={(current bounding box.center)},thick]
\draw(0,0) node[draw=black,fill=white,very thick,rounded corners,baseline=center, inner xsep=0.5cm] (a){$f$};
\draw [color = blue]([xshift=0cm]a.north) --++ (0, 1);
\draw [color = blue]([xshift=-0.15cm]a.north) --++ (0, 1);
\draw [color = blue]([xshift=0.3cm]a.south) --++ (0, -1);
\draw [color = blue]([xshift=0.15cm]a.south) --++ (0, -1);
\draw ([xshift=-0.3cm]a.south) to [out=-90, in=90] node[pos=1, below] {$x$} ([xshift=-0.6cm, yshift=-1cm]a.south);
\node[above] at ([yshift=1cm]a.north) {$\Lambda_n$};
\node[below] at ([xshift=0.3cm, yshift=-1cm]a.south) {$\Lambda_n$};

\node[] (b) at ([xshift=1.25cm]a) {$\otimes$};

\draw([xshift=1.25cm]b) node[draw=black,fill=white,very thick,rounded corners,baseline=center, inner xsep=0.5cm] (a){$g$};
\draw [color = blue]([xshift=-0.15cm]a.north) --++ (0, 1);
\draw [color = blue]([xshift=-0.3cm]a.north) --++ (0, 1);
\draw [color=red] ([xshift=0.3cm]a.north) --++ (0.3, 1);
\draw ([xshift=-0.3cm]a.south) to [out=-90, in=90] node[pos=1, below] {$y$} ([xshift=-0.6cm, yshift=-1cm]a.south);
\node[above] at ([xshift=-0.35cm, yshift=1cm]a.north) {$\Lambda_n$};
\node[above] at ([xshift=0.8cm, yshift=1cm]a.north) {$R_n \setminus \Lambda_n$};

\node[] (b) at ([xshift=1.5cm]a) {$\mapsto$};

\draw([xshift=1.5cm, yshift=0.5cm]b) node[draw=black,fill=white,very thick,rounded corners,baseline=center, inner xsep=0.5cm] (c){$f$};
\draw [color = blue]([xshift=0cm]c.north) --++ (0, 0.5) coordinate (d);
\draw [color = blue]([xshift=-0.15cm]c.north) --++ (0, 0.5);
\draw [color = blue]([xshift=0.15cm]c.south) --++ (0, -0.5);
\draw [color = blue]([xshift=0.3cm]c.south) --++ (0, -0.5);
\node[above] at ([xshift=-0.1cm, yshift=0.5cm]c.north) {$\Lambda_n$};

\draw([xshift=0.4cm, yshift=-1cm]c) node[draw=black,fill=white,very thick,rounded corners,baseline=center, inner xsep=0.5cm] (a){$g$};
\draw ([xshift=-0.3cm]a.south) to [out=-90, in=90] node[pos=1, below] {$y$} ([xshift=-0.6cm, yshift=-0.5cm]a.south) coordinate (f);

\draw ([xshift=-0.3cm]c.south) to [out=-90, in=90] node[pos=1, below] {$x$} ([xshift=-0.5cm]f);
\draw [color=red] ([xshift=0.3cm]a.north) to ([xshift=1cm]d) coordinate (e);
\node[above] at ([xshift=-0.1cm]e) {$R_n \setminus \Lambda_n$};

\end{tikzpicture}
    \end{center}
is unitary.
\end{lemma}
\begin{proof}
    That $\phi_n$ is isometric can be verified using the graphical calculus:

\begin{center}
\begin{tikzpicture}[baseline={(current bounding box.center)},thick]
\draw(0,0) node[draw=black,fill=white,very thick,rounded corners,baseline=center, inner xsep=0.3cm] (a){$f_2$};
\draw ([yshift=1cm]a.north) node[draw=black,fill=white,very thick,rounded corners,baseline=center, inner xsep=0.3cm] (b){$f_1^*$};
\draw([xshift=0.45cm, yshift=-1cm]a) node[draw=black,fill=white,very thick,rounded corners,baseline=center, inner xsep=0.3cm] (c){$g_2$};
\draw([xshift=0.45cm, yshift=1cm]b) node[draw=black,fill=white,very thick,rounded corners,baseline=center, inner xsep=0.3cm] (d){$g_1^*$};

\draw ([xshift=-0.2cm]c.south) to [out=-155, in=-90] node[pos=1, left, color=black] {$\overline y$} ([xshift=-1.5cm, yshift=0.35cm]a.north) coordinate (e);
\draw (e) to [out=90, in=155] ([xshift=-0.2cm]d.north);

\draw ([xshift=-0.2cm]a.south) to [out=-135, in=-90] node[pos=1, left, color=black] {$\overline x$} ([xshift=-0.9cm, yshift=0.35cm]a.north) coordinate (f);
\draw (f) to [out=90, in=135] ([xshift=-0.2cm]b.north);

\draw [color = blue] ([xshift=0cm]a.north) to node[pos=0.6, left, color=black] {$\Lambda_n$} ([xshift=0cm]b.south);
\draw [color = blue] ([xshift=0.15cm]a.north) to ([xshift=0.15cm]b.south);
\draw [color = blue] ([xshift=0.15cm]b.north) to ([xshift=-0.3cm]d.south);
\draw [color = blue] ([xshift=0.3cm]b.north) to ([xshift=-0.15cm]d.south);
\draw [color = blue] ([xshift=0.15cm]a.south) to ([xshift=-0.3cm]c.north);
\draw [color = blue] ([xshift=0.3cm]a.south) to ([xshift=-0.15cm]c.north);
\draw [color = red] ([xshift=0.3cm]c.north) to node[pos=0.5, right, color=black] {$R_n \setminus \Lambda_n$} ([xshift=0.3cm]d.south);

\node[] at ([xshift=-4.8cm, yshift=0.75cm]a) {$\langle \phi_n(f_1 \otimes g_1) \, | \, \phi_n(f_2 \otimes g_2) \rangle_{K_{x \otimes y}} = $};

\node[] at ([xshift=4.15cm, yshift=0.65cm]a) {$ = \langle \langle f_2 | f_1 \rangle \cdot g_1 \, | \, g_2 \rangle_{K_y}$};
\end{tikzpicture}
\end{center}

    To see that $\phi_n$ is surjective, observe that $X^{\otimes R_n}$ containing a copy of every simple object implies that there exist $e_1,\dots,e_\ell\in\cB(X^{\otimes R_n},\overline{x}\otimes X^{\otimes R_n})$ such that $\sum_{j=1}^\ell e_j\circ (e_j)^*=\Id_{\overline{x}\otimes X^{\otimes R_n}}$. So, every $g \in K_{x\otimes y}^{R_n}$ can be written as follows:
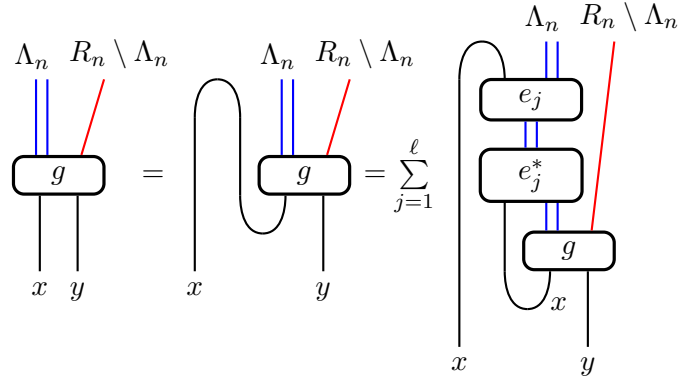
\begin{center}
    \begin{tikzpicture}[baseline={(current bounding box.center)},thick]
\draw(0,0) node[draw=black,fill=white,very thick,rounded corners,baseline=center, inner xsep=0.5cm] (a){$g$};
\draw [color=blue] ([xshift=-0.3cm]a.north) --++ (0, 1);
\draw [color=blue] ([xshift=-0.15cm]a.north) --++ (0, 1);
\draw [color=red] ([xshift=0.3cm]a.north) --++ (0.3, 1);
\node[above] at ([xshift=-0.35cm, yshift=1cm]a.north) {$\Lambda_n$};
\node[above] at ([xshift=0.8cm, yshift=1cm]a.north) {$R_n \setminus \Lambda_n$};
\draw ([xshift=-0.25cm]a.south) --node[pos=1, below] {$x$}++ (0, -1);
\draw ([xshift=0.25cm]a.south) --node[pos=1, below] {$y$}++ (0, -1);

\node[] (z) at ([xshift=1.25cm]a) {$=$};

\draw([xshift=2cm]z) node[draw=black,fill=white,very thick,rounded corners,baseline=center, inner xsep=0.5cm] (a){$g$};
\draw [color=blue] ([xshift=-0.3cm]a.north) --++ (0, 1);
\draw [color=blue] ([xshift=-0.15cm]a.north) --++ (0, 1);
\draw [color=red] ([xshift=0.3cm]a.north) --++ (0.3, 1);
\node[above] at ([xshift=-0.35cm, yshift=1cm]a.north) {$\Lambda_n$};
\node[above] at ([xshift=0.8cm, yshift=1cm]a.north) {$R_n \setminus \Lambda_n$};
\draw ([xshift=0.25cm]a.south) --node[pos=1, below] {$y$}++ (0, -1) coordinate (g);
\draw ([xshift=-0.25cm]a.south) to [out=-90, in=0] ([xshift=-0.55cm, yshift=-0.5cm]a.south) coordinate (b);
\draw (b) to [out=180, in=-90] ([xshift=-0.85cm]a.south) coordinate (c);
\draw (c) to [out=90, in=-90] ([xshift=-0.85cm, yshift=0.5cm]a.north) coordinate (d);
\draw (d) to [out=90, in=0] ([xshift=-1.15cm, yshift=1cm]a.north) coordinate (e);
\draw (e) to [out=180, in=90] ([xshift=-1.45cm, yshift=0.5cm]a.north) coordinate (f);
\draw (f) to node[pos=1, below] {$x$} ([xshift=-1.7cm]g);

\node[] (z) at ([xshift=1.25cm]a) {$= \sum \limits_{j = 1}^\ell$};

\draw([xshift=2.25cm, yshift=-1cm]z) node[draw=black,fill=white,very thick,rounded corners,baseline=center, inner xsep=0.5cm] (a){$g$};
\draw [color=red] ([xshift=0.3cm]a.north) --++ (0.3, 2.5);
\node[above] at ([xshift=-0.35cm, yshift=2.5cm]a.north) {$\Lambda_n$};
\node[above] at ([xshift=0.8cm, yshift=2.5cm]a.north) {$R_n \setminus \Lambda_n$};
\draw ([xshift=0.25cm, ]a.south) --node[pos=1, below] {$y$}++ (0, -1) coordinate (g);

\draw ([xshift=-0.25cm]a.south) to [out=-90, in=0] node[pos=0.7, right] {$x$} ([xshift=-0.55cm, yshift=-0.5cm]a.south) coordinate (b);
\draw (b) to [out=180, in=-90] ([xshift=-0.85cm]a.south) coordinate (c);
\draw(c) to [out=90, in=-90] ([xshift=-0.85cm, yshift=1cm]a.north) coordinate (d);
\draw([yshift=0.5cm]d) to [out=90, in=-90] ([xshift=-0.85cm, yshift=2cm]a.north) coordinate (y);
\draw (y) to [out=90, in=0] ([xshift=-1.15cm, yshift=2.5cm]a.north) coordinate (e);
\draw (e) to [out=180, in=90] ([xshift=-1.45cm, yshift=2cm]a.north) coordinate (f);
\draw (f) to node[pos=1, below] {$x$} ([xshift=-1.7cm]g);

\draw [color=blue] ([xshift=-0.3cm]a.north) --++ (0, 1);
\draw [color=blue] ([xshift=-0.3cm, yshift=1.75cm]a.north) --++ (0, 0.75);
\draw [color=blue] ([xshift=-0.15cm]a.north) --++ (0, 1);
\draw [color=blue] ([xshift=-0.15cm, yshift=1.75cm]a.north) --++ (0, 0.75);

\draw([xshift=-0.5cm, yshift=1cm]a) node[draw=black,fill=white,very thick,rounded corners,baseline=center, inner xsep=0.5cm] (h){$e_j^*$};
\draw([yshift=1cm]h) node[draw=black,fill=white,very thick,rounded corners,baseline=center, inner xsep=0.5cm] (i){$e_j$};

\draw [color=blue] ([xshift=-0.075cm]h.north) to ([xshift=-0.075cm] i.south);
\draw [color=blue] ([xshift=0.075cm]h.north) to ([xshift=0.075cm] i.south);

\end{tikzpicture}
\end{center}
\end{proof}

\begin{construction}
    We now use $\{\phi_n\}$ to construct our tensorator $\phi:K_x({\Lambda})\boxtimes K_y\to K_{x\otimes y}$. Each $f\in \cB(x\otimes X^{\otimes R_n}, X^{\otimes R_n})$ includes into $\cB(x\otimes X^{\otimes R_{n+1}}, X^{\otimes R_{n+1}})$ as follows:
    
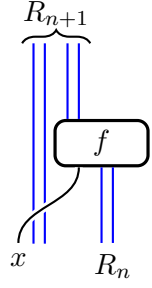
\begin{center}
        \begin{tikzpicture}[baseline={(current bounding box.center)},thick]
\draw(0,0) node[draw=black,fill=white,very thick,rounded corners,baseline=center, inner xsep=0.5cm] (a){$f$};
\draw [color = blue]([xshift=-0.3cm]a.north) --++ (0, 1) coordinate (c);
\draw [color = blue]([xshift=-0.45cm]a.north) --++ (0, 1);
\draw [color = blue]([xshift=0cm]a.south) --++ (0, -1);
\draw [color = blue]([xshift=0.15cm]a.south) --++ (0, -1);
\draw [color = blue]([xshift=-0.75cm, yshift=1cm]a.north) --++ (0, -2.65);
\draw [color = blue]([xshift=-0.9cm, yshift=1cm]a.north) coordinate (b) --++ (0, -2.65);
\draw [knot] ([xshift=-0.3cm]a.south) to [out=-90, in=90] ([xshift=-1.1cm, yshift=-1cm]a.south) coordinate (z);
\node[below] at (z) {$x$};
\node[below] at ([xshift=0.15cm, yshift=-1cm]a.south) {$R_n$};
\draw [decorate,decoration={brace,amplitude=5pt,mirror,raise=0ex}]
  ([xshift=0.15cm]c) -- ([xshift=-0.15cm]b) node[midway,yshift=1em]{$R_{n+1}$};
\end{tikzpicture}
    \end{center}

    This inclusion process allows $f$ to act on the dense subset $\bigcup_j K_1^{R_j}\subset K_1$. We thus obtain embeddings 
    \[
    \iota_n:\Hom_{\mathfrak{A}(R_n\setminus \Lambda)}(K_1^{R_n},K_x^{R_n})\to\Hom_{A_{\Lambda^c}}(K_1,K_x)
    \]
    Let $p_n\in B(K_x)$ be the projection onto $K_x(R_n)$. For any $F\in\Hom_{A_{\Lambda^c}}(K_1,K_x)$, set 
    \[
    f_n=p_n\circ F|_{K_1(R_n)}\in \Hom_{\mathfrak{A}(R_n\setminus \Lambda)}(K_1(R_n),K_x(R_n)).
    \]
    Then $\iota_n(f_n)$ is a (bounded) sequence that converges to $F$ strongly. 
    By Lemma \ref{lemma:dense-subsets-of-Connes-fusion}, $\bigcup_n\Im(\iota_n)\otimes K_y(R_n)$ is dense in $K_x(\Lambda)\boxtimes K_y$ (here, we are identifying $K_y(R_n)$ with its image in $K_y$). Note that $\{\phi_n\}$ is compatible with our inclusion maps, i.e., the below diagram commutes
    \[\begin{tikzcd}[column sep=2em, row sep=1em]
	{\Hom_{\mathfrak{A}(R_n\setminus \Lambda)}(K_1^{R_n},K_x^{R_n})\otimes_{\mathfrak{A}(R_n)}K_y^{R_n}} && {\Hom_{\mathfrak{A}(R_{n+1}\setminus \Lambda)}(K_1^{R_{n+1}},K_x^{R_{n+1}})\otimes_{\mathfrak{A}(R_{n+1})}K_y^{R_{n+1}}} \\
	\\
	{K_{x\otimes y}^{R_n}} && {K_{x\otimes y}^{R_{n+1}}}
	\arrow[from=1-1, to=1-3]
	\arrow["{\phi_n}"', from=1-1, to=3-1]
	\arrow["{\phi_{n+1}}", from=1-3, to=3-3]
	\arrow[from=3-1, to=3-3]
\end{tikzcd}\]
    We thus obtain a unitary $\phi:K_x({\Lambda})\boxtimes K_y\to K_{x\otimes y}$ that restricts to $\phi_n$ for each square $R_n$ (that is, on each subspace $\Hom_{\mathfrak{A}(R_n\setminus \Lambda)}(K_1(R_n),K_x(R_n))\otimes_{\mathfrak{A}(\Lambda_n)}K_y(R_n)$).
\end{construction}

\begin{proposition}
    The unitary $\phi:K_x({\Lambda})\boxtimes K_y\to K_{x\otimes y}$ intertwines the net of algebras $\cA$.
\end{proposition}
\begin{proof}
    Since each $\phi_n$ intertwines $\mathfrak{A}(\Lambda_n)$ and $\mathfrak{A}(R_n\setminus\Lambda)$, it follows that $\phi$ intertwines $A_\Lambda$ and $A_{\Lambda^c}$. Moreover, $\phi$ is natural with respect to the choice of cone $\Lambda$. That is, given an inclusion of cones $\Lambda\subset\Delta$, the diagram below commutes.
    \[\begin{tikzcd}[column sep=.5em, row sep=.5em]
	{K_x({\Lambda})\boxtimes K_y} && {K_x({\Delta})\boxtimes K_y} \\
	& {K_{x\otimes y}}
	\arrow[from=1-1, to=1-3]
	\arrow[from=1-1, to=2-2]
	\arrow[from=1-3, to=2-2]
\end{tikzcd}\]
This can be seen by moving to orderings of $R_n$ that send points in $\Lambda$ to the left, followed by points in $\Delta\setminus\Lambda$ to the middle, and points in $R_n\setminus\Delta$ to the right. 
We achieve this arrangement of strands by composing with the appropriate element of the braid group, or its inverse.
\end{proof}

Since $\mathsf{Hilb}\langle X \rangle$ allows for multiplicity spaces, we need to take them into account when defining our tensorator. The following lemma essentially allows us to move the multiplicity spaces out of the way. We prove this for a general net of algebras $\cA \subset B(H)$ since the proof does not rely upon the structure of the braided categorical.
\begin{lemma}
    Let $K$ and $L$ be superselection sectors for the net of algebras $\cA$. 
    Let $H_1$ and $H_2$ be Hilbert spaces, and consider the superselection sectors $H_1 \otimes K$ and $H_2 \otimes L$ where each $H_i$ acts as a multiplicity space. 
    Then $(H_1 \otimes K)(p) \boxtimes (H_2 \otimes L)$ and $(H_1 \otimes H_2) \otimes (K(p)\boxtimes L)$ are isomorphic superselection sectors via the unitary $$\theta((\xi \otimes f) \boxtimes (\zeta \otimes \eta))=(\xi \otimes \zeta) \otimes (f \boxtimes \eta),$$
    where $\xi\otimes f$ denotes the map $\chi\mapsto \xi\otimes f(\chi)$.
\end{lemma}

\begin{proof}
    We first claim that elements of the form $(\xi \otimes f) \boxtimes (\zeta \otimes \eta)$ have dense span in $(H_1 \otimes K)(p) \boxtimes (H_2 \otimes L)$, where $\xi \in H_1, \zeta \in H_2, f \in \Hom_{A_p'}(H, K)$, and $\eta \in L$. To see this, given an orthonormal basis $e_1,e_2,\dots$ of $H_1$, note that $\sum_{i=1}^\infty p_{e_i} \otimes \id_{K}$ converges to $\id_{H_1 \otimes K}$ in the strong operator topology. Furthermore, each of these projections is a morphism in $\SSS$ since $\cA$ acts solely on the second tensor factor. Therefore, by Lemma~\ref{lemma:dense-subsets-of-Connes-fusion}, elements of the stated form are norm-dense as required.

    We next compute the adjoint of the operator $\xi \otimes f \in \Hom(H, H_1 \otimes K)$. For any $\chi \in H$ and $\alpha \otimes \beta \in H_1 \otimes K$, 
    \begin{align*}
        \langle (\xi \otimes f)(\chi)| \alpha \otimes \beta\rangle &= \langle \xi \otimes f(\chi)|\alpha \otimes \beta\rangle\\
        &= \langle \xi|\alpha\rangle \langle f(\chi)|\beta\rangle\\
        &=\langle \chi| \langle \xi|\alpha\rangle f^*(\beta)\rangle,
    \end{align*}
    so $(\xi \otimes f)^*(\alpha\otimes\beta)=\langle \xi|\alpha\rangle f^*(\beta)$. We now check that $\theta$ is preserves inner products on a dense subset:
    \begin{align*}
        \langle (\xi_1\otimes\zeta_1) \otimes (f_1 \boxtimes \eta_1)|(\xi_2 \otimes \zeta_2)\otimes (f_2 \boxtimes \eta_2)\rangle &=\langle \xi_1 |\xi_2\rangle \langle \zeta_1|\zeta_2\rangle \langle \eta_1| (f_1^* f_2)\cdot \eta_2\rangle\\
        &=\langle \zeta_1|\zeta_2\rangle \langle \eta_1|  f_1^*f_2\cdot \langle \xi_1|\xi_2\rangle \eta_2\rangle \\
        &=\langle \zeta_1 \otimes \eta_1|(\langle \xi_1|\xi_2\rangle f_1^*f_2) \cdot(\zeta_2 \otimes \eta_2)\rangle\\
        &= \langle \zeta_1 \otimes \eta_1|((\xi_1\otimes f_1)^*\circ (\xi_2\otimes f_2)) \cdot(\zeta_2 \otimes \eta_2)\rangle\\
        &= \langle(\xi_1 \otimes f_1) \boxtimes (\zeta_1 \otimes \eta_1)|(\xi_2 \otimes f_2) \boxtimes (\zeta_2 \otimes \eta_2)\rangle.
    \end{align*}
    Thus, $\theta$ is well-defined and extends to a unitary on $(H_1 \otimes K)(p) \boxtimes (H_2 \otimes L)$. Furthermore, $\theta$ is easily seen to be an intertwiner in $\SSS$.
\end{proof}

We now define the tensorator on isotypic components for the functor $\cK$ as follows and extend by linearity:
\begin{equation}
\label{eq:TensoratorForEquivalence}
(\id_{H_x \otimes H_y}\otimes \mu_{x,y})\circ \theta \colon (H_1 \otimes \cK_x)(\Lambda)\boxtimes (H_2 \otimes \cK_y) \to (H_1 \otimes H_2) \otimes \cK_{x \otimes y}.
\end{equation}
Finally, we can prove Thm \ref{thm:B}.

\begin{theorem}
When endowed with tensorator \eqref{eq:TensoratorForEquivalence}, the $\rmW^*$-equivalence \eqref{eq:EquivalenceForBraidedNet} from \cite[Theorem~4.9]{2506.19969} 
\[
\cK: \Hilb\langle X\rangle^{\rm op} \to \mathsf{SSS}(A)
\]
is monoidal and braided, and therefore an equivalence of braided monoidal categories.
\end{theorem}

\begin{proof}
We begin by showing monoidality. 
Because the map $\theta$ moves the multiplicity spaces $H_x$ and $H_y$ to the left hand side, it suffices to show that the following diagram commutes when we take the multiplicity spaces to be $\mathbb C$.

\begin{center}
    \begin{tikzcd}[column sep=scriptsize]
        ((H_x \otimes \cK_x)(\Lambda) \boxtimes (H_y \otimes \cK_y))(\Lambda) \boxtimes (H_z\otimes \cK_z) \arrow[r, "\alpha_{\SSS}"] \arrow[d, "\phi \boxtimes \id"'] & (H_x \otimes \cK_x)(\Lambda) \boxtimes ((H_y \otimes \cK_y)(\Lambda) \boxtimes (H_z \otimes \cK_z)) \arrow[d, "\id\otimes \phi"'] \\
        ((H_x \otimes H_y)\otimes \cK_{xy})(\Lambda)\boxtimes (H_z \otimes \cK_z) \arrow[d, "\phi"] & (H_x \otimes \cK_x)(\Lambda)\boxtimes (H_y\otimes H_z \otimes \cK_{yz}) \arrow[d, "\phi"']\\
        ((H_x \otimes H_y)\otimes H_z) \otimes \cK_{(xy)z} \arrow[r, "\cK(\alpha_{\Hilb\langle X \rangle})"]& (H_x \otimes (H_y \otimes H_z))\otimes \cK_{x(yz)}
    \end{tikzcd}
\end{center}

Thus, when $H_x=H_y=H_z=\mathbb C$, we see that the images of a simple tensor $(f_x\boxtimes f_y(-))\boxtimes g_z$ under the top right and bottom left paths, for intertwiners $f_x \in \Hom_{\mathfrak{A}(R_n\setminus \Lambda)}(\cK_1(R_n), \cK_x(R_n))$ and $ f_y \in \Hom_{\mathfrak{A}(R_n)\setminus \Lambda)}(\cK_1(R_n), \cK_y(R_n))$, and $g_z \in \cK_z(R_n)$, are both

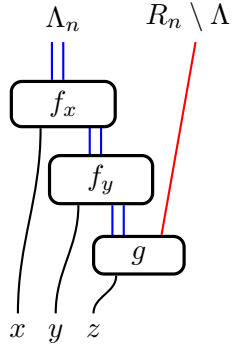
\begin{center}
    \begin{tikzpicture}[baseline={(current bounding box.center)}, thick]
\draw(0,0) node[draw=black,fill=white,very thick,rounded corners,baseline=center, inner xsep=0.5cm] (c){$f_y$};
\draw [color = blue]([xshift=0cm]c.north) --++ (0, 0.5) coordinate (d);
\draw [color = blue]([xshift=-0.15cm]c.north) --++ (0, 0.5);
\draw [color = blue]([xshift=0.15cm]c.south) --++ (0, -0.5);
\draw [color = blue]([xshift=0.3cm]c.south) --++ (0, -0.5);
\node[above] at ([xshift=-0.1cm, yshift=0.5cm]c.north) {$\Lambda_n$};

\draw([xshift=0.5cm, yshift=-1cm]c) node[draw=black,fill=white,very thick,rounded corners,baseline=center, inner xsep=0.5cm] (a){$g$};

\draw([xshift=-0.5cm, yshift=1cm]c) node[draw=black,fill=white,very thick,rounded corners,baseline=center, inner xsep=0.5cm] (b){$f_x$};

\draw ([xshift=-0.3cm]a.south) to [out=-90, in=90] node[pos=1, below] {$z$} ([xshift=-0.6cm, yshift=-0.5cm]a.south) coordinate (f);

\draw ([xshift=-0.3cm]c.south) to [out=-90, in=90] node[pos=1, below] {$y$} ([xshift=-0.5cm]f);
\draw ([xshift=-0.3cm]b.south) to [out=-90, in=90] node[pos=1, below] {$x$} ([xshift=-1cm]f);

\draw [color = blue]([xshift=0cm]b.north) --++ (0, 0.5) coordinate (f);
\draw [color = blue]([xshift=-0.15cm]b.north) --++ (0, 0.5);

\draw [color=red] ([xshift=0.3cm]a.north) to ([xshift=1.75cm]f) coordinate (e);
\node[above] at ([xshift=-0.1cm]e) {$R_n \setminus \Lambda$};
\node[above] at (f) {$\Lambda_n$};
\end{tikzpicture}
\end{center}
Because elements of the above form are norm-dense, we conclude that the diagram commutes.

To check the appropriate unitality diagrams for the tensorator, we can again take the multiplicity space $H_x$ to be $\mathbb C$. Furthermore, since $\cK_1$ is the monoidal unit for $\SSS$, the diagram we need to verify commutes is
\begin{center}
    \begin{tikzcd}
        \cK_1(\Lambda)\boxtimes (H_x \otimes \cK_x) \arrow[d, "\lambda_{\SSS}"] \arrow[dr, "\mu"] & \\
        H_x \otimes \cK_x & H_x \otimes \cK_{1 \otimes x} \arrow[l, "\cK(\lambda_{\mathsf{Hilb}\langle X \rangle})"] 
    \end{tikzcd}
\end{center}
Taking a simple tensor $a \boxtimes g_x$ (with $a \in \mathfrak{A}(R_n \cap \Lambda) \cong \Hom_{\mathfrak{A}(R_n\setminus \Lambda)}(\cK_1(R_n), \cK_1 (R_n))$ and $g_x\in \cK_x(R_n)$), we see its image under either path is

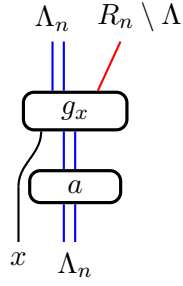
\begin{center}
    \begin{tikzpicture}[baseline={(current bounding box.center)}, thick]
        \draw(0,0) node[draw=black,fill=white,very thick,rounded corners,baseline=center, inner xsep=0.5cm] (c){$g_x$};
\draw [color = blue]([xshift=-0.15cm]c.north) --++ (0, 0.65) coordinate (d);
\draw [color = blue]([xshift=-0.3cm]c.north) to ([xshift=-0.15cm]d);
\draw [color = blue]([xshift=0cm]c.south) --++ (0, -0.5);
\draw [color = blue]([xshift=-0.15cm]c.south) --++ (0, -0.5);
\node[above] at ([xshift=-0.15cm]d) {$\Lambda_n$};

\draw([xshift=0cm, yshift=-1cm]c) node[draw=black,fill=white,very thick,rounded corners,baseline=center, inner xsep=0.5cm] (a){$a$};
\draw [color = blue]([xshift=0cm]a.south) --++ (0, -0.5) coordinate (e);
\draw [color = blue]([xshift=-0.15cm]a.south) --++ (0, -0.5);
\node[below] at (e) {$\Lambda_n$};

\draw ([xshift=-0.45cm]c.south) to [out=-90, in=90]  ([xshift=-0.75cm]a) coordinate (z);
\draw (z) to node[pos=1, below] {$x$} ([xshift=-0.75cm, yshift=-0.5cm]a.south);
\draw [color=red] ([xshift=0.3cm]c.north) to ([xshift=0.75cm]d) coordinate (e);
\node[above] at ([xshift=0.25cm]e) {$R_n \setminus \Lambda$};
    \end{tikzpicture}
\end{center}

\noindent Again, by density, we conclude that the diagram commutes. 
The second unitality axiom is similar and left to the reader.

Finally, to verify the braided equivalence, we check the following diagram commutes.
\begin{center}
    \begin{tikzcd}
        (H_x \otimes \cK_x)(\Lambda)\boxtimes (H_y \otimes \cK_y) \arrow[r, "\beta_{\SSS}"] \arrow[d, "\phi"]& (H_y \otimes \cK_y)(\Lambda) \boxtimes (H_x \otimes \cK_x) \arrow[d, "\phi"] \\
        (H_x \otimes H_y) \otimes \cK_{xy} \arrow[r, "\cK(\beta_{\mathsf{Hilb}\langle X \rangle})"] & (H_y \otimes H_x) \otimes \cK_{yx}
    \end{tikzcd}
\end{center}
We can once again assume $H_x=H_y=\mathbb C$. Taking a splitting $(\Lambda^r,\Lambda^s)$ of $\Lambda$, we know that the space $\cK_x(\Lambda)\boxtimes \cK_y$ will be norm-densely spanned by elements of the form $f_x^r \boxtimes f_y^s(g_1)$, where $f_x^r \in \Hom_{\mathfrak{A}(R_n \setminus \Lambda^r)}(\cK_1(R_n),\cK_x(R_n))$, $f_y^s \in \Hom_{\mathfrak{A}(R_n \setminus \Lambda^s)}(\cK_1(R_n),\cK_y(R_n))$, and $g_1 \in \cK_1(R_n)$. We then have that the two possible images of this simple tensor under the diagram are
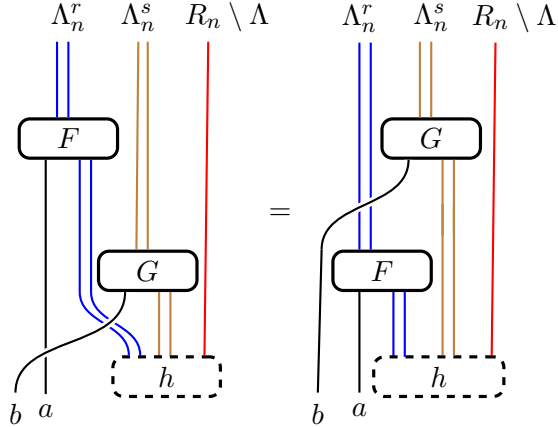
\[
    \begin{tikzpicture}[baseline={(current bounding box.center)}, thick]
   \draw(0,0) node[draw=black,fill=white,very thick,rounded corners,baseline=center, inner xsep=0.5cm] (a){$F$};
   \draw [color=blue] ([xshift=0cm]a.north) --++ (0, 1) coordinate (d);
   \draw [color=blue] ([xshift=-0.15cm]a.north) --++ (0, 1);
   \draw [color=blue] ([xshift=0.3cm]a.south) --++ (0, -1.75);
   \draw [color=blue] ([xshift=0.15cm]a.south) --++ (0, -1.75);
   \draw ([xshift=-0.3cm]a.south) to [out=-90, in=90] node[pos=1, below, color = black, opacity = 1] {$a$} ([xshift=-0.3cm, yshift=-3.1cm]a.south);
   \node[above] at ([yshift=1cm]a.north) {$\Lambda_n^r$};

   \draw([xshift=1.05cm, yshift=-1.75cm]a) node[draw=black,fill=white,very thick,rounded corners,baseline=center, inner xsep=0.5cm] (a){$G$};
   \draw [color=brown] (a.north) to ([xshift=1.05cm]d);
   \draw [color=brown] ([xshift=-0.15cm]a.north) to ([xshift=0.925cm]d) coordinate (w);
   \draw [color=brown] ([xshift=0.3cm]a.south) --++ (0, -1);
   \draw [color=brown] ([xshift=0.15cm]a.south) --++ (0, -1);
   \draw([xshift=0.25cm, yshift=-1.4cm]a) node[draw=black, dashed, fill=white,very thick,rounded corners,baseline=center, inner xsep=0.6cm] (c){$h$};
   \draw [color=red] ([xshift=0.5cm]c.north) to ([xshift=1.85cm]d) coordinate (v);
   \draw [color=blue] ([xshift=-0.35cm]c.north) to [out=90, in=-90] ([xshift=-0.75cm]a.south) coordinate (x);
   \draw [color=blue] ([xshift=-0.5cm]c.north) to [out=90, in=-90] ([xshift=-0.9cm]a.south) coordinate (y);
   \draw [knot] ([xshift=-0.3cm]a.south) to [out=-90, in=90] ([xshift=-1.75cm, yshift=-1.33cm]a.south) coordinate (z);
   \node[below] at (z) {$b$};
   \node[above] at (w) {$\Lambda_n^s$};
   \node[above] at ([xshift=0.25cm]v) {$R_n \setminus \Lambda$};

   \node[] (v) at ([xshift=1.75cm, yshift=0.75cm]a) {$=$};
   
   \draw([xshift=2cm, yshift=1cm]v) node[draw=black,fill=white,very thick,rounded corners,baseline=center, inner xsep=0.5cm] (a){$G$};
   \draw [color=brown] ([xshift=0cm]a.north) --++ (0, 1);
   \draw [color=brown] ([xshift=-0.15cm]a.north) --++ (0, 1);
   \draw [color=brown] ([xshift=0.3cm]a.south) --++ (0, -1);
   \draw [color=brown] ([xshift=0.15cm]a.south) --++ (0, -1) coordinate (y);
   \node[above] at ([yshift=1cm]a.north) {$\Lambda_n^s$};

   \draw([xshift=-0.65cm, yshift=-1.75cm]a) node[draw=black,fill=white,very thick,rounded corners,baseline=center, inner xsep=0.5cm] (f){$F$};
   \draw [color=blue] ([xshift=-0.15cm]f.north) to ([xshift=-0.95cm, yshift=2.55cm]y);
   \draw [color=blue] ([xshift=-0.3cm]f.north) to ([xshift=-1.1cm, yshift=2.55cm]y) coordinate (w);
   \draw [color=blue] ([xshift=0.3cm]f.south) --++ (0, -1);
   \draw [color=blue] ([xshift=0.15cm]f.south) --++ (0, -1);
   \draw ([xshift=-0.3cm]f.south) to [out=-90, in=90] ([xshift=-0.3cm, yshift=-1.33cm]f.south) coordinate (z);
   \node[below] at (z) {$a$};
   \draw[knot] ([xshift=-0.3cm]a.south) to [out=-90, in=90] ([xshift=-0.8cm]f.north) coordinate (g);
   \draw (g) to node[pos=1, below] {$b$} ([xshift=-0.55cm]z);
   \draw([xshift=0.75cm, yshift=-1.4cm]f) node[draw=black, dashed, fill=white,very thick,rounded corners,baseline=center, inner xsep=0.75cm] (c){$h$};
   \draw [color=brown] ([xshift=0.05cm]c.north) --++ (0,2);
   \draw [color=brown] ([xshift=0.2cm]c.north) --++ (0,2);
   \draw [color=red] ([xshift=0.7cm]c.north) to ([xshift=1.8cm]w) coordinate (x);
   \node[above] at (w) {$\Lambda_n^r$};
   \node[above] at ([xshift=0.25cm]x) {$R_n \setminus \Lambda$};
\end{tikzpicture}
\]
which are equal. By density, we conclude that the diagram commutes.
\end{proof}

\appendix
\section{Pentagon and Triangle Identities}\label{app:pentagon-triangle}
Here we give a direct proof of Proposition~\ref{prop:sss-monoidal}:
the associator 
    \begin{align*}
        \alpha_{A, B, C} \colon (A \boxtimes B(p)) \boxtimes C(p) &\to A \boxtimes (B \boxtimes C(p))(p) \\
        (\eta_a \otimes g_b) \otimes g_c &\mapsto \eta_a \otimes (g_B(-) \otimes g_c)
    \end{align*}
is a unitary intertwiner in $\SSS$ which satisfies the pentagon relation.
\begin{proof}
    The associator is manifestly natural. We show that it is unitary. In the following diagram, $\alpha_{A,B,C}$ is a priori defined on a dense subspace of $(A \boxtimes B(p))\boxtimes C(p)$.
\[\begin{tikzcd}[column sep=1.5em, row sep=1em]
	{(A \boxtimes B(p)) \boxtimes C(p)} && {A \boxtimes (B \boxtimes C(p))(p)} \\
	\\
	{(A(p') \boxtimes B(p)) \boxtimes C(p)} && {A(p') \boxtimes (B \boxtimes C(p))(p)} \\
	\\
	{(A(p') \boxtimes B) \boxtimes C(p)} && {A(p') \boxtimes(B \boxtimes C(p))}
	\arrow["{\alpha_{A, B, C}}", from=1-1, to=1-3]
	\arrow["{\cL_{p', p} \boxtimes \id_{C(p)}}", from=3-1, to=1-1]
	\arrow["{\cR_{p', p} \boxtimes \id_{C(p)}}"', from=3-1, to=5-1]
	\arrow["{\cL_{p', p}}"', from=3-3, to=1-3]
	\arrow["{\cR_{p', p}}", from=3-3, to=5-3]
	\arrow[from=5-1, to=5-3]
\end{tikzcd}\]

    All arrows in this diagram are unitary. The bottom arrow is given by reparenthesization. For $(f_A \otimes \xi \otimes g_B) \otimes g_C \in (A(p') \boxtimes B(p)) \boxtimes C(p)$, we compute clockwise:
    \begin{align*}
        (f_A \otimes \xi \otimes g_B) \otimes g_C &\mapsto (f_A(\xi) \otimes g_B) \otimes g_C \mapsto f_A(\xi) \otimes (g_B(-) \otimes g_C) \\
        &\mapsto f_A \otimes \xi \otimes (g_B(-)\otimes g_c) \mapsto f_A \otimes (g_B(\xi) \otimes g_C) \\
        &\mapsto (f_A \otimes g_B(\xi)) \otimes g_C \mapsto (f_A \otimes \xi \otimes g_B) \otimes g_C,
    \end{align*}
    so the diagram commutes. Observe that for disjoint $r,s\in\cP$, the reparenthesization map 
    \[
    \widetilde{\alpha}_{A,B,C,r,s}:(A(r)\boxtimes B)\boxtimes C(s)\to A(r)\boxtimes(B\boxtimes C(s))
    \]
    intertwines $A_r$ and $A_s$, and satisfies naturality in disjoint pairs $(r,s)$. That is, if $(r_1,s_1),(r_2,s_2)$ are two disjoint pairs in $\cP$ and $r_1\leq r_2$ and $s_1\leq s_2$, then the diagram below commutes.
    \[\begin{tikzcd}[column sep=2.5em, row sep=1em]
	{(A(r_1)\boxtimes B)\boxtimes C(s_1)} && {A(r_1)\boxtimes (B\boxtimes C(s_1))} \\
	\\
	{(A(r_2)\boxtimes B)\boxtimes C(s_2)} && {A(r_2)\boxtimes (B\boxtimes C(s_2))}
	\arrow["{\widetilde{\alpha}_{A,B,C,r_1,s_1}}", from=1-1, to=1-3]
	\arrow[from=1-1, to=3-1]
	\arrow[from=1-3, to=3-3]
	\arrow["{\widetilde{\alpha}_{A,B,C,r_2,s_2}}"', from=3-1, to=3-3]
\end{tikzcd}\]
Therefore, the reparenthesization map (and hence the associator) intertwines $\cA$. Morevoer, $\alpha_{A,B,C}$ extends by continuity and is itself unitary.

    Furthermore, the pentagon axiom holds.
\[\begin{tikzcd}[column sep=0, row sep=1em]
	& {((A \boxtimes B(p))\boxtimes C(p)) \boxtimes D(p)} & \\
	\\
	{(A \boxtimes (B \boxtimes C(p))(p)) \boxtimes D(p)} && {(A \boxtimes B(p)) \boxtimes (C \boxtimes D(p))(p)} \\
	\\
	{A \boxtimes ((B \boxtimes C(p)) \boxtimes D(p))(p)} && {A \boxtimes (B \boxtimes (C \boxtimes D(p))(p))(p)}
	\arrow["{\alpha_{A,B,C} \boxtimes \id}"', from=1-2, to=3-1]
	\arrow["{\alpha_{A \boxtimes B(p), C, D}}", from=1-2, to=3-3]
	\arrow["{\alpha_{A, B \boxtimes C(p), D}}"', from=3-1, to=5-1]
	\arrow["{\alpha_{A, B, C \boxtimes D(p)}}", from=3-3, to=5-3]
	\arrow["{\id \boxtimes \alpha_{B, C, D}}"', from=5-1, to=5-3]
\end{tikzcd}\]
    For an element of $(((A \boxtimes B)\boxtimes C) \boxtimes D)$, going through the left path is the computation
    \begin{align*}
    (((\eta_A \otimes g_B) \otimes g_C) \otimes g_D) &\mapsto (\eta_A \otimes (g_B(-)\otimes g_C)) \otimes g_D \mapsto \eta_A \otimes ((g_B(-) \otimes g_C))(-) \otimes g_D \\
    &\mapsto \eta_A \otimes (\alpha_{B, C, D} \circ ((g_B(-) \otimes g_C) (-) \otimes g_D))
    \end{align*}
    On the other hand, going through the right path is the computation 
    $$
    ((\eta_A \otimes g_B) \otimes g_C) \otimes g_D \mapsto (\eta_A \otimes g_B) \otimes (g_C(-) \otimes g_D) \mapsto \eta_A \otimes (g_B(-) \otimes (g_C(-) \otimes g_D)).
    $$
    To see that these agree, we check that $\alpha_{B, C, D} \circ ((g_B(-) \otimes g_C)(-) \otimes g_D) = (g_B(-) \otimes (g_C(-) \otimes g_D))$. 
    Evaluating at $\xi \in A \boxtimes (B \boxtimes (C \boxtimes D(p))(p))(p)$,

    \begin{align*}
    (g_B(-)\otimes (g_C(-) \otimes g_D))(\xi) 
    &= 
    g_B(\xi) \otimes (g_C(-) \otimes g_D) 
    \\&= 
    \alpha_{B, C, D} \circ ((g_B(\xi) \otimes g_C) \otimes g_D) 
    \\&= 
    \alpha_{B, C, D} \circ ((g_B(-) \otimes g_C)(\xi) \otimes g_D) 
    \\&= 
    \alpha_{B, C, D} \circ ((g_B(-) \otimes g_C)(-) \otimes g_D)(\xi).
    \end{align*} 
    
    The unitors are clearly natural. 
    One can quickly check that they are unitary intertwiners and satisfy the triangle axiom.
\end{proof}
\section{Hexagon Identities}\label{app:hexagon-identities}
In this section, we directly verify the following hexagon axiom diagrams commute for the braiding $\beta$ on $(\SSS,\boxtimes_p)$.
\begin{equation}\label{braidhex1}
    \begin{tikzcd}[column sep=1em, row sep=1em]
        & A \boxtimes (B \boxtimes C(p))(p) \arrow[r, "\beta"] & (B \boxtimes C(p)) \boxtimes A(p) \arrow[dr, "\alpha"] & \\
        (A \boxtimes B(p)) \boxtimes C(p) \arrow[ur, "\alpha"] \arrow[dr, "\beta \boxtimes \Id_C"'] & & & B \boxtimes (C \boxtimes A(p))(p) \\
        & (B \boxtimes A(p)) \boxtimes C(p) \arrow[r, "\alpha"']& B \boxtimes (A \boxtimes C(p))(p) \arrow[ur, "\Id_B \boxtimes \beta"'] &
    \end{tikzcd}
\end{equation}

\begin{equation}\label{braidhex2}
    \begin{tikzcd}[column sep=1em, row sep=1em]
        & (A \boxtimes B(p)) \boxtimes C(p) \arrow[r, "\beta"] & C \boxtimes (A \boxtimes B(p))(p) \arrow[dr, "\alpha^{-1}"] & \\
        A \boxtimes (B \boxtimes C(p))(p) \arrow[ur, "\alpha^{-1}"] \arrow[dr, "\Id_A \boxtimes \beta"'] & & & (C \boxtimes A(p))\boxtimes B(p) \\
        & A \boxtimes (C \boxtimes B(p))(p) \arrow[r, "\alpha^{-1}"']& (A \boxtimes C(p)) \boxtimes B(p) \arrow[ur, "\beta \boxtimes \Id_B"'] &
    \end{tikzcd}
\end{equation}
We use $\eta$'s with subscripts to denote vectors, and $g$'s with subscripts to denote $A_p'$-linear Homs. 
We first verify that diagram~\ref{braidhex1} commutes. 
Consider the element $(\eta_A \otimes g_B) \otimes g_C \in (A \boxtimes B(p))\boxtimes C(p)$. Choose a unitary $u \in \Hom_{A_p'}(H, A)$. We then have, tracing along the top half of \ref{braidhex1}, 
\begin{align*}
(\eta_A \otimes g_B) \otimes g_C &\xmapsto{\alpha} \eta_A \otimes (g_B(-)\otimes g_C) \\
&\xmapsto{\beta}(g_B(-)\otimes g_C)(u^*\eta_A) \otimes u \\
&= (g_B(u^*\eta_A)\otimes g_C)\otimes u \\
&\xmapsto{\alpha} g_B(u^*\eta_A) \otimes (g_C(-)\otimes u)
\end{align*}
On the other hand, if we trace along the bottom half of $\ref{braidhex1}$, we observe that
\begin{align*}
(\eta_A \otimes g_B)\otimes g_C &\xmapsto{\beta \boxtimes \Id_C}(g_B(u^*\eta_A) \otimes u)\otimes g_C \\
&\xmapsto{\alpha}g_B(u^*\eta_A) \otimes (u(-) \otimes g_C) \\
&\xmapsto{\Id_B \boxtimes \beta} g_B(u^*\eta_A) \otimes (\beta \circ (u(-)\otimes g_C) \\
&=g_B(u^*\eta_A) \otimes (g_C(-)\otimes u)
\end{align*}
Where the last equality follows from the fact that $\beta(u(\xi) \otimes g_C) = g_C(\xi) \otimes u$ for any $\xi \in H$.
\par Now, to verify diagram~\ref{braidhex2}, choose a splitting $r_1,r_2 \leq r$ (where $r$ is used in the definition of the braiding), and pick unitaries $v \in \Hom_{A_{r_1'}}(H,A)$ and $w \in \Hom_{A_{r_2'}}(H,B)$. We claim that $v(-) \otimes w$ is a unitary in $\Hom_{A_{r'}}(H, A \boxtimes B(p))$. Note that since $r_1,r_2 \leq r$, $v(-) \otimes w$ is certainly $A_{r'}$ -linear. Thus, to verify that it is a unitary, we compute
$$\langle (v(-)\otimes w)^*( \xi_A \otimes f_B)| \xi \rangle = \langle \xi_A \otimes f_B| v(\xi) \otimes w \rangle = \langle \pi_p^A(w^*f_b) \xi_A| v(\xi)\rangle = \langle v^* \pi_p^A(w^*f_b)\xi_A| \xi\rangle$$
from which we conclude that $(v(-)\otimes w)^*(\xi_A \otimes f_B)= v^*\pi_p^A(w^*f_B)\xi_A$. We now see that
$$(v(-)\otimes w)^*(v(-)\otimes w)(\xi)= (v(-)\otimes w)^*(v(\xi) \otimes w) =v^*\pi_p^A(w^*w)v(\xi)=\xi$$
and
\begin{align*}
    (v(-)\otimes w)(v(-)\otimes w)^*(\xi_A \otimes f_B)&= (v(-)\otimes w)(v^*\pi_p^A(w^*f_B)\xi_A)=v(v^*\pi_p^A(w^*f_B)\xi_A)\otimes w \\
    &= \xi_A \otimes ww^*f_B = \xi_A \otimes f_B
\end{align*}
Returning to diagram~\ref{braidhex2}, we may use the fact that the span of simple tensors $\eta_A \otimes (g_B(-) \otimes g_C)$ is dense in $A \boxtimes (B \boxtimes C(p))(p)$ as $\alpha$ is unitary. Thus commutativity of the diagram may be verified with elements of this form. Furthermore, because the inclusion map $B(r_2)\hookrightarrow B(p)$ is unitary (and so has dense range), we may furthermore suppose that $g_B$ is an $A_{r_2'}$-linear map. Proceeding along the top third of the diagram, we have
\begin{align*}
    \eta_A \otimes (g_B(-) \otimes g_C) &\xmapsto{\alpha^{-1}} (\eta_A \otimes g_B) \otimes g_C \\
    &\xmapsto{\beta}g_C((v(-)\otimes w)^*(\eta_A \otimes g_B)) \otimes (v(-)\otimes w) \\
    &=g_Cv^*\pi_p^A(w^*g_B)\eta_A \otimes (v(-) \otimes w)
\end{align*}
On the other hand, we may proceed along the bottom half of the diagram (and then apply the upward map $\alpha$), and we see
\begin{align*}
    \eta_A \otimes (g_B(-) \otimes g_C) &\xmapsto{\id_A \otimes \beta} \eta_A \otimes (\beta \circ (g_B(-) \otimes g_C)) \\
    &= \eta_A \otimes (g_Cw^*g_B(-) \otimes w)\\
    &\xmapsto{\alpha^{-1}} (\eta_A \otimes g_Cw^*g_B) \otimes w\\
    &\xmapsto{\beta \boxtimes \Id_B} (g_Cw^*g_B(v^*\eta_A) \otimes v) \otimes w \\
    &\xmapsto{\alpha} g_Cw^*g_B(v^*\eta_A) \otimes (v(-) \otimes w)
\end{align*}
Noting that $w^*g_B$ is an $A_{r_2'}$-linear endomorphism of $H$, it must belong to $A_{r_2}\subset A_p$, and as $v^*$ is $A_p$-linear, we have
$$g_Cw^*g_B(v^*\eta_A) \otimes (v(-) \otimes w)=g_Cv^*\pi_p^A(w^*g_B)\eta_A \otimes (v(-) \otimes w)$$
as desired.

\newcommand{\etalchar}[1]{$^{#1}$}

\end{document}